\documentclass[a4paper,12pt]{amsart}

\usepackage{amsmath,amssymb,amsthm}
\usepackage{graphicx,stmaryrd}
\usepackage{verbatim}
\usepackage{hyperref}

\newcommand{\CC}{\mathbb{C}}

\newcommand{\PP}{\mathbb{P}} 
\newcommand{\RR}{\mathbb{R}}

\newcommand{\ZZ}{\mathbb{Z}}

\newcommand{\cB}{\mathcal{B}}

\newcommand{\cF}{\mathcal{F}}

\newcommand{\cH}{\mathcal{H}}

\newcommand{\cO}{\mathcal{O}}

\newcommand{\tc}{\mathrm{c}}

\renewcommand{\a}{\alpha}
\newcommand{\D}{\Delta} 
\renewcommand{\d}{\delta}

\newcommand{\g}{\gamma} 
\renewcommand{\L}{\Lambda}

\renewcommand{\l}{\lambda}
\renewcommand{\b}{\beta}

\newcommand{\om}{\omega} 
\renewcommand{\S}{\Sigma} 
\newcommand{\s}{\sigma}

\newcommand{\eps}{\varepsilon}

\renewcommand{\max}{\mathrm{max}}

\newcommand{\el}{\langle} 
\newcommand{\er}{\rangle}
\newcommand{\tr}{\mathrm{tr}}

\newcommand{\bh}{\mathbf{h}}
\newcommand{\sd}{\triangle}

\newcommand{\fk}{{\sc fk}}

\newcommand{\gks}{{\sc gks}}

\newcommand{\lra}{\leftrightarrow}

\renewcommand{\b}{\beta}
\newcommand{\oo}{\infty}

\newcommand{\sm}{\setminus}

\newcommand{\es}{\varnothing}
\newcommand{\se}{\subseteq}

\newcommand{\crit}{\mathrm{c}} 
\newcommand{\fp}{\mathfrak{p}}

\newcommand{\z}{\zeta}

\newcommand{\one}{\hbox{\rm 1\kern-.27em I}}

\newcommand{\Ob}{[0,\b)^\fp}

\newtheoremstyle{slthm}
     {}
     {\baselineskip}
     {\slshape}
     {\parindent}
     {\scshape}
     {.}
     { }
     {}

\theoremstyle{slthm}
\newtheorem{definition}{Definition}[section]
\newtheorem{theorem}[definition]{Theorem}
\newtheorem{proposition}[definition]{Proposition}
\newtheorem{lemma}[definition]{Lemma}

\newtheorem{remark}[definition]{Remark}

\title[Infrared bound]{Infrared bound and mean-field behaviour\\
in the quantum Ising model}
\author{Jakob E. Bj\"ornberg}
\thanks{Department of Mathematics,
Uppsala University, Box 256, 751 05 Uppsala, Sweden, 
Phone +46(0)18-471 3106,
e-mail: jakob@math.uu.se}
\date{\today}

\begin{document}
\maketitle

\begin{abstract}
We prove an infrared bound for the transverse field
Ising model.  This bound is stronger than the previously
known infrared bound for the model, and allows us to
investigate mean-field behaviour.  As an application
we show that the critical exponent $\g$
for the susceptibility attains its mean-field value 
$\g=1$ in dimension at least 4 (positive temperature),
respectively 3 (ground state), with logarithmic corrections
in the boundary cases.
\end{abstract}

\section{Introduction}
\label{intro_sec}

Infrared bounds were originally developed in the 1970's as a method
for proving the existence of phase 
transitions~\cite{dls,froehlich_etal78,fss}.  They
were subsequently also used to establish mean-field behaviour in
high-dimen\-sional 
spin systems, in the sense that certain critical exponents attain their
mean-field values~\cite{aiz82,af,ag,sokal79}.  
For establishing phase
transitions the method is applicable to models
where the spins have a continuous symmetry group, 
such as the classical Heisenberg and $O(n)$ models;
see~\cite{biskup_rp} for a review.  
In proving mean-field
behaviour, the infrared bound is useful because it implies the
finiteness of the `bubble diagram' in sufficiently high dimension.  
Similar bounds appear also in the analysis of high-dimensional
percolation models~\cite{hara_slade}.

The method of infrared bounds was first developed for classical 
spin systems, but it was quickly extended to quantum 
models.  In the quantum setting,
the method was successful for
proving the existence of a phase transition for a range of models,
including the Heisenberg antiferromagnet~\cite{dls}.  However, no
attention has yet been paid to its implications for mean-field behaviour
in quantum models.  The main objective of this article is to
establish an infrared bound which can be used to investigate
mean-field behaviour in the quantum setting.

The  model we will consider in this article is the
\emph{transverse field Ising model} on the integer lattice $\ZZ^d$.
Let $\L\se\ZZ^d$ be finite.  In the language of quantum spin systems,
the transverse field Ising model in the volume $\L$ has Hamiltonian
\begin{equation}\label{h1}
H_\L=-\l\sum_{xy\in\L}\s_x^{(3)}\s_y^{(3)}-
\d\sum_{x\in\L}\s_x^{(1)}
\end{equation}
acting on the Hilbert space $\cH_\L=\bigotimes_{x\in\L}\CC^2$.  Here
the first sum is over all (unordered)
nearest neighbour sites in $\L$ and the second
sum is over all single sites; 
\begin{equation*}
\s^{(3)}=\begin{pmatrix} 1 & 0\\ 0 & -1\end{pmatrix}
\quad\mbox{and}\quad
\s^{(1)}=\begin{pmatrix} 0 & 1\\ 1 & 0\end{pmatrix}
\end{equation*}
are the spin-$\tfrac{1}{2}$ Pauli matrices; and $\l$ and $\d$ denote the
spin-coupling and transverse field intensities, respectively.  
The model was introduced in~\cite{lsm} and has been widely
studied since.

Let $\b>0$ be a fixed real number (known as the `inverse temperature')
and define the finite volume, positive temperature 
state $\el\cdot\er_{\L,\b}$ by
\begin{equation}\label{q_state_eq}
\el Q\er_{\L,\b}=\frac{\tr(e^{-\b H_\L}Q)}{\tr(e^{-\b H_\L})},
\end{equation}
where $Q$ is a suitable observable (matrix).  One may then define 
the finite-volume \emph{ground state} $\el\cdot\er_{\L,\oo}$, 
as well as \emph{infinite-volume} states $\el\cdot\er_\b$ by
\begin{equation*}
\el \cdot\er_{\L,\oo}=\lim_{\b\uparrow\oo}\el\cdot\er_{\L,\b},
\quad \el \cdot\er_\b=\lim_{\L\uparrow\ZZ^d}\el\cdot\er_{\L,\b},
\quad\mbox{for } 0<\b\leq\oo.
\end{equation*}
See for example~\cite{akn} for more information about this.  
It is known~\cite{ginibre69}
that for each $\d>0$ and $0<\b\leq\oo$ 
there is a critical value $\l_\crit=\l_\crit(\d,\b)$
of $\l$ which marks the point below which the 
\emph{two-point correlation function}
\begin{equation}\label{tp}
c(x,y)=\el\s_x^{(3)}\s_y^{(3)}\er_\b
\end{equation} 
vanishes as $|x-y|\rightarrow\infty$, and above which it does not.
This critical point may also be characterized in terms of
the vanishing/non-vanishing of the spontaneous 
$\s^{(3)}$-magnetization,
or of the finiteness/non-finiteness of the 
susceptibility~\cite{bjogr}.

For the classical Ising model ($\d=0$ and $\l=1$ 
in~\eqref{h1}), the infrared bound
is an upper bound on the Fourier transform
\begin{equation*}
\hat c(k)=\sum_{x\in\ZZ^d} c(0,x)e^{ik\cdot x}, \quad k\in(-\pi,\pi]^d
\end{equation*}
of the two-point function~\eqref{tp}.
Here $k\cdot x$ denotes the usual scalar product in $\RR^d$.
The classical infrared bound states that
\begin{equation}\label{cirb}
\hat c(k)\leq \frac{1}{2\b \hat L(k)},
\end{equation}
where
$\hat L(k)=\sum_{j=1}^d(1-\cos(k_j))$
is the Fourier transform of the Laplacian on $\ZZ^d$.
For quantum models, infrared bounds are typically stated in terms of
the \emph{Duhamel two-point function}
\begin{equation}\label{dtf}
b(x)=(\s^{(3)}_0,\s^{(3)}_x):=\frac{1}{\tr(e^{-\b H_\L})}
\int_0^\b\tr(e^{-(\b-t) H_\L}\s^{(3)}_xe^{-t H_\L}\s^{(3)}_0)\,dt
\end{equation}
rather than the usual two-point function~\eqref{tp}.
For the transverse field Ising model, it follows 
as a special case of~\cite[Theorem~4.1]{dls} that
\begin{equation}\label{dtf_irb}
\hat b(k)\leq \frac{1}{2\l\hat L(k)}.
\end{equation}
(Note that our $\hat b(k)$ and $H_\L$ differ from
the quantities in~\cite{dls} by factors $\b$ and $\l$,
respectively.)

The main result of this article, Theorem~\ref{irb_thm},
strengthens~\eqref{dtf_irb}.  This will allow us to
compute the
critical exponent $\g$ for the susceptibility in dimension
$d\geq4$ (finite temperature) respectively $d\geq3$ (ground state);
see Theorem~\ref{mf_thm}.  
Before stating the main result 
we describe a graphical representation of 
the transverse field Ising model which is of fundamental importance to
our analysis.

\subsection{Graphical representation}
\label{gr_ssec}

It is well-known~\cite{akn,aizenman_nacht}
that the transverse field Ising model
on $\L$ possesses a `path integral representation' which expresses it
as a type of classical Ising model on the continuous space
$\L\times[0,\beta)$.  This may be expressed as follows.
In what follows $\b<\oo$, although conclusions about the
ground state may be obtained by letting $\b\rightarrow\oo$.

Let $\Ob$ denote the circle of length $\b$,
formally defined as $\Ob=\{e^{2\pi it/\b}:t\in\RR\}$.
Usually we identify $\Ob$ with its
parameterization for $t\in[0,\b)$;
the superscript $\fp$ serves as a reminder that
the set is `periodic'.  We let $N$ be an integer and 
throughout the article let 
$\L=\L_N=(\ZZ/2N)^d$ be the $d$-dimensional
torus of side $2N$.

Let $E[\cdot]$ denote a probability measure
governing the following:
\begin{enumerate}
\item a collection $D=(D_x: x\in\L)$ of independent
Poisson processes of intensity 
$\d$ on $[0,\b)$, conditioned on the number of points
 $|D_x|$ being even for each $x\in\L$;  and
\item a collection $\xi=(\xi_x: x\in \L)$ of independent
random variables, taking values $0$ or $1$ with
probability $1/2$ each,
\end{enumerate}
such that $D$ and $\xi$ are independent of each other.
Write $\s(x,t)=(-1)^{\xi_x+|D_x\cap[0,t]|}$ for the right-continuous
function of $t$ which changes between $-1$ and $+1$ at the
points of $D_x$ and takes the value $(-1)^{\xi_x}$ at $t=0$.
Note that $\s$ is well-defined as a function on $\Ob$. 
Write $\S_\L$ for the set of such $\s(x,t)$.
This set may be endowed with a natural 
sigma-field $\cF_\L$, generated by the projections
\begin{equation}\label{proj_eq}
\s\mapsto \big(\s(x_1,t_1),\dotsc,\s(x_n,t_n)\big)\quad
\mbox{for } n\geq 1,\; x_j\in\L,\; t_j\in[0,\b).
\end{equation}
Let 
\begin{equation}\label{pf1}
Z^\b_\L=E\Big[\exp\Big(\l\sum_{x\sim y}\int_0^\b\s(x,t)\s(y,t)\,dt\Big)
\Big],
\end{equation}
where the sum is over nearest neighbours $x,y\in\L$.
\begin{definition}\label{st_def}
The space--time Ising measure $\mu^\b_\L$ is the probability
measure on $(\S_\L,\cF_\L)$ given by
\begin{equation*}
\mu^\b_\L(f)=\frac{1}{Z^\b_\L}
E\Big[f(\s)\exp\Big(\l\sum_{x\sim y}
\int_0^\b \s(x,t)\s(y,t)\,dt\Big)\Big]
\end{equation*}
for each bounded, measurable test function $f:\S_\L\rightarrow\RR$.
\end{definition}
(Here, and in what follows, $\mu^\b_\L(f)$ denotes the expectation of
$f$ under the measure $\mu^\b_\L$.)
The measure $\mu^\b_\L$ corresponds with the state 
$\el\cdot\er_{\L,\b}$ in
the sense that for any finite set $A\se\L$, we have the identity
\begin{equation}\label{qi_st}
\Big\el\prod_{x\in A}\s_x^{(3)}\Big\er_{\L,\b}=
\mu^\b_\L\Big(\prod_{x\in A}\s(x,0)\Big).
\end{equation}
In very brief terms, this correspondence is
based on applying the Lie--Trotter product formula
to the operator $e^{-\b H_\L}$ and evaluating the trace in
the $\s^{(3)}$-basis.
See~\cite{akn} and references therein, and also~\cite{ioffe_geom}
for information about similar graphical representations.
There exist limits of the measures $\mu^\b_\L$ as
$N\rightarrow\oo$ and/or $\b\rightarrow\oo$, and a
suitable version of~\eqref{qi_st} holds in
infinite volume, also for the ground state.  


\subsection{Main results}

We now outline the main results of this article, 
saving more detailed statements and explanations for the
relevant later sections.  Our main results concern the
\emph{Schwinger function}:
\begin{equation}\label{tp2}
c^\b_\L((x,s),(y,t)):=
\mu^\b_\L\big(\s(x,s)\s(y,t)\big),\qquad(0\leq s,t<\b).
\end{equation}
Note that this may alternatively be expressed as 
\[
c^\b_\L((x,s),(y,t))=\frac{1}{\tr(e^{-\b H_\L})}
\tr(e^{-(\b-t+s) H_\L}\s^{(3)}_ye^{-(t-s) H_\L}\s^{(3)}_x),
\]
as may be seen using a Lie--Trotter expansion as
for~\eqref{qi_st};
cf.~\eqref{q_state_eq} and~\eqref{dtf}.
For $k\in\frac{2\pi}{2N}\L$ and $l\in\frac{2\pi}{\b}\ZZ$ let
\begin{equation*}
\hat c^\b_\L(k,l)=\sum_{x\in\L} \int_0^\b c^\b_\L((0,0),(x,t))
e^{ik\cdot x}e^{ilt}\,dt
\end{equation*}
be the Fourier transform of~\eqref{tp2}.
The following is the main result of this article;  it holds
for finite volume and positive temperature,
but has implications for infinite volume and ground state
which we discuss below. 
\begin{theorem}[Infrared bound]\label{irb_thm}
Let $k\in\frac{2\pi}{2N}\L$ and $l\in\frac{2\pi}{\b}\ZZ$.  
Then
\begin{equation*}
\hat c^\b_\L(k,l)\leq \frac{48}{2\l \hat L(k)+l^2/2\d}.
\end{equation*}
\end{theorem}
This result is proved in Section~\ref{irb_sec}.
In fact, we will prove the slightly stronger bound
\[
\hat c_\L^\b(k,l)\leq 
\frac{2\l\hat L(k)+48l^2/2\d}{(2\l \hat L(k)+l^2/2\d)^2},
\]
see~\eqref{messy_bound}.  Setting $l=0$  gives 
\begin{equation}\label{dls_eq}
\sum_{x\in\L}\Big(\int_0^\b\mu^\b_\L(\s(0,0)\s(x,t))\,dt\Big)
e^{ik\cdot x} \leq \frac{1}{2\l\hat L(k)}.
\end{equation}
The integral in~\eqref{dls_eq} equals the Duhamel two-point function
$b(x)=(\s^{(3)}_0,\s^{(3)}_x)$ of~\eqref{dtf}.
Thus~\eqref{dls_eq} is~\cite[Theorem~4.1]{dls} in the special case
of the transverse field Ising model.  It also reduces to the classical
bound~\eqref{cirb} when $\d=0$ and $\l=1$.
We now describe some applications of Theorem~\ref{irb_thm}.

Let $0<\b\leq\oo$ be fixed.  For each $\l<\l_\crit$ 
there is a unique infinite-volume limit 
$\mu^\b$ of the measures $\mu^\b_\L$.
We let
\begin{equation}\label{sus_eq}
\chi=\chi(\d,\l,\b):=
\sum_{x\in\ZZ^d}\int_{-\b/2}^{\b/2}\mu^\b(\s(0,0)\s(x,t))\,dt
\end{equation}
denote the \emph{susceptibility}.  
(We have chosen the bounds $-\b/2$, $\b/2$ in the integral rather than
$0$, $\b$ to get the correct range also for the case $\b=\oo$; for
$\b<\oo$ the two choices are equivalent.)
To motivate this choice of name, note that
\begin{equation*}
\begin{split}
\sum_{x\in\L}\int_{-\b/2}^{\b/2}\mu^\b_\L(\s(0,0)\s(x,t))\,dt
&=\sum_{x\in\L}(\s_0^{(3)},\s_x^{(3)})\\
&=\frac{d}{d\nu}\Big[\frac{\tr(e^{-\b H_\L^\nu}\s_0^{(3)})}
{\tr(e^{-\b H_\L^\nu})}\Big]_{\nu=0}
\end{split}
\end{equation*}
where
\begin{equation*}
H_\L^\nu=-\l\sum_{xy\in\L}\s_x^{(3)}\s_y^{(3)}-
\d\sum_{x\in\L}\s_x^{(1)}-\nu\sum_{x\in\L}\s_x^{(3)}.
\end{equation*}
It is known that
$\chi$ is finite for
$\l<\l_\crit$ and diverges as $\l\uparrow\l_\crit$; 
see~\cite{bjogr}.  The critical
exponent $\g$ may be defined by the expected critical behaviour 
\begin{equation*}
\chi(\l)\sim (\l_\crit-\l)^{-\g}\quad\mbox{as }\l\uparrow\l_\crit
\mbox{ with $\d$ fixed}.
\end{equation*}
For the classical Ising model, it was proved 
in~\cite{aiz82,ag} that $\g$ exists and equals 1
for $d\geq 4$ (with logarithmic corrections when $d=4$).
One may also consider the speed of divergence 
of $\chi(\d,\l)$ as $(\d,\l)$ approaches the
critical curve $(\d,\l_\crit(\d))$ along any
straight line.  As an application of Theorem~\ref{irb_thm},
we will prove the following result.  We let 
$\rho(\d,\l)=\sqrt{(\d-\d_0)^2+(\l-\l_\crit(\d_0))}$ denote the
distance from $(\d,\l)$ to a specified point
$(\d_0,\l_\crit(\d_0))$ on the critical curve.
\begin{theorem}\label{mf_thm}
Let $\d_0>0$ and let 
$(\d,\l)$ approach $(\d_0,\l_\crit(\d_0))$ along
any straight line strictly inside the quadrant
$\{(\d,\l):\d>\d_0,\l<\l_\crit(\d_0)\}$.
There are finite constants $c_1,c_2$, depending 
on $d$, $\b$ and the slope of the line of approach, such that
the following holds.
\begin{enumerate}
\item Suppose that either $\b<\oo$ and $d>4$, or
$\b=\oo$ and $d>3$.  Then   
\begin{equation*}
c_1\rho(\d,\l)^{-1}\leq\chi(\d,\l)\leq c_2\rho(\d,\l)^{-1}
\end{equation*}
as $\rho(\d,\l)\downarrow 0$.
\item Suppose that either $\b<\oo$ and $d=4$, or
$\b=\oo$ and $d=3$.   Then
\begin{equation*}
c_1\rho(\d,\l)^{-1}\leq\chi(\d,\l)\leq -c_2\log\rho(\d,\l)/\rho(\d,\l)
\end{equation*}
as $\rho(\d,\l)\downarrow 0$.
\end{enumerate}
\end{theorem}
Theorem~\ref{mf_thm} is proved in Section~\ref{mf_sec};
in fact we prove slightly more, see
Propositions~\ref{lb_prop} and~\ref{ub_prop}.
In proving Theorem~\ref{mf_thm} we are led
to study the \emph{bubble-diagram}
\begin{equation}\label{b2}
B=B(\d,\l,\b)=\sum_{x\in\ZZ^d}\int_{-\b/2}^{\b/2}
\mu^\b(\s(0,0)\s(x,t))^2\,dt.
\end{equation}
(The term `bubble-diagram' comes from analogy with a related quantity
which appears in  the study of mean-field
behaviour in the classical Ising model~\cite{aiz82,af,ag}.)
Theorem~\ref{irb_thm}, together with the Plancherel identity, 
allows us to deduce upper bounds on $B$ 
(Lemma~\ref{bub_lem2}).  In addition to such bounds we also
require  new differential
inequalities (Lemma~\ref{diff_lem}).
The method of proof would also give the critical
exponent $\g=1$ for approach to criticality along 
lines with constant $\d$ or constant $\l$,
\emph{subject to} first proving additional 
differential inequalities.  See Remark~\ref{gamma_rk}.


\subsection{Discussion}\label{disc_sec}

It is interesting to note that we obtain the `classical'
critical exponent value $\g=1$ also at the 
`quantum critical point' $\b=\oo$, $\l=\l_\crit(\d)$.
Our analysis only deals with approach to criticality
with temperature (hence $\b$) kept fixed, so our results
do not rule out the possibility of a different critical
exponent value for approach to the quantum critical point
with varying temperature, as described 
in~\cite{sachdev_article}.

In the classical Ising model, the type of methods
used in this article can only give conclusions
about critical exponents down to dimension $d=4$.
We are able to obtain conclusions about the case $d=3$
essentially because the `imaginary time representation'
described in Section~\ref{gr_ssec} maps the quantum
Ising model in $d$ dimensions onto a classical model
in $d+1$ dimensions.  This intuition is in some sense 
only valid in the case $\b=\oo$ when the `imaginary time axis' 
is unbounded.

In the classical Ising model it was also possible
to use the infrared bound and differential inequalities
to determine critical exponents for the magnetization~\cite{af}.
It is to be expected that Theorem~\ref{irb_thm}
can be used to obtain similar results also for the
quantum Ising model.

Note that the arguments used to prove
Theorem~\ref{mf_thm} require
bounds on $B$ as defined above, \emph{not} on
the (one might think more natural) quantity
\begin{equation*}
\sum_{x\in\ZZ^d}\el\s^{(3)}_0\s^{(3)}_x\er_{\b}^2
=\sum_{x\in\ZZ^d}\mu^\b(\s(0,0)\s(x,0))^2,
\end{equation*}
which does not feature in this work.

It is worth remarking that we approach Theorem~\ref{irb_thm} by
working directly in the continuous set-up of 
Definition~\ref{st_def}.  A natural alternative would be to work
with the discrete approximation of the partition function~\eqref{pf1}
inherent in the Lie--Trotter product formula on which
Definition~\ref{st_def} is based.  In this way certain technicalities
associated with working with continuous `time' may be avoided;  on the
other hand other issues to do with discrete approximation would be
introduced.  This discrete approximation is closely related to a
way of expressing the space--time Ising model as a (weak) limit of
classical Ising models (cf.~\cite[Section~2.2.2]{bjo_phd}),
each of which obeys a bound of the form~\eqref{cirb}.  
Unfortunately, simply taking the limit in the corresponding bound
gives only~\eqref{dtf_irb} and not the full bound of 
Theorem~\ref{irb_thm}.

Finally, although we have chosen to focus entirely on the
nearest-neighbour transverse field Ising model, it seems likely that
our arguments can be extended to other reflection positive models.
For example, it seems straightforward to extend Theorem~\ref{irb_thm}
to other interactions than nearest neighbour, such as the
Yukawa and power law potentials described
in~\cite[Section~3]{biskup_rp}.  Many of the arguments
in Section~\ref{irb_sec} apply when the
definition $\s(x,t)=(-1)^{\xi_x+|D_x\cap[0,t]|}$ in the graphical
representation~\eqref{pf1} is modified to
$\s(x,t)=\psi_x(-1)^{|D_x\cap[0,t]|}$ for more general
$\psi_x\in\RR^n$.  The proofs of the differential inequalities in
Section~\ref{mf_sec}, and hence Theorem~\ref{mf_thm}, rely on the
`random-parity representation' (a relative of the random-current
representation) and so are quite Ising-specific.  However, certain
extensions of these results
to more general translation-invariant
interactions are most likely possible, as for the classical case
treated in~\cite{aiz82}.

\subsection*{Acknowledgements}
The author would like to thank the following.
Geoffrey Grimmett for 
introducing him to the subject, and for many valuable
discussions in the early phases of the project.  Svante Janson for valuable
feedback on a draft of this article, and in particular for
suggesting an improved proof of Lemma~\ref{binary_lem}.
And finally the anonymous referees for many insightful and helpful
comments.


\section{The infrared bound}
\label{irb_sec}

In this section we prove the main result of this
article, Theorem~\ref{irb_thm}.  First we present
more detailed notation.

We write $\one\{A\}$ for the indicator of the
event $A$, taking value $1$ if $A$ occurs and
$0$ otherwise.
Let $N\geq1$ be an integer, and let $\L=(\ZZ/2N)^d$ be a 
torus in $d$ dimensions.
Thus $\L$ is the graph whose vertex set is the
set of vectors
$x=(x_1,\dotsc,x_d)\in\{0,1,\dotsc,2N-1\}^d$, 
and whose adjacency relation
$\sim$ is given by:  $x\sim y$ if there is $j\in\{1,\dotsc,d\}$
such that (a) $x_j-y_j\equiv 1 \mod{2N}$,
and (b) $x_k=y_k$ for all $k\neq j$.
Write $L$ for the \emph{graph Laplacian} of $\L$:
\begin{equation}\label{laplacian}
L(x,y)=d\one\{x=y\}-\frac{1}{2}\one\{x\sim y\},
\qquad x,y\in\L.
\end{equation}
We write $Lu$
for the function $\L\rightarrow\CC$ given by
the matrix-vector product
\begin{equation}\label{mat_vec}
(Lu)(x)=\sum_{y\in\L}L(x,y)u(y).
\end{equation}
For $u,v:\L\rightarrow\CC$ we write $\el u,v\er$
for the usual vector inner product,
\begin{equation}\label{dotprod}
\el u,v\er=\sum_{x\in\L} u(x)v(x)\in\CC.
\end{equation}
Note that we do not conjugate the second argument.

Throughout this section $\b$ will be finite and fixed.
Recall that $\Ob$ denotes the circle of length $\b$.
A function $f:\Ob\rightarrow\RR$ can be
thought of either as a function $\RR\rightarrow\RR$
which is periodic with period $\b$, or
simply as a function $[0,\b)\rightarrow\RR$.
We will usually take the latter viewpoint, taking care to
remember, for example, that $f$ is continuous
only if, in particular, the limits
$\lim_{t\uparrow\b}f(t)$ and 
$\lim_{t\downarrow0}f(t)$
exist and are equal, and similarly for differentiability
and other analytic properties.
The antiderivative of a measurable $f:\Ob\rightarrow\RR$
is the function $F:\Ob\rightarrow\RR$ given by
\begin{equation*}
F(t)=\int_0^t f(s)\,ds 
\mbox{ for }t\in[0,\b).
\end{equation*}

If $\bh$ is a function
$\L\times\Ob\rightarrow\CC$,  we will 
often use the 
notation $\bh=(h(x,t):x\in\L,t\in\Ob)$.
For each $x\in\L$ we then write $h(x,\cdot)$
for the function $t\mapsto h(x,t)$,
and for each $t\in\Ob$ we write $h(\cdot,t)$
for the function $x\mapsto h(x,t)$.  Note that
$h(\cdot,t):\L\rightarrow\CC$, so the notation
in~\eqref{mat_vec} and~\eqref{dotprod} 
applies to $h(\cdot,t)$.
We say that a function $\bh:\L\times\Ob\rightarrow\RR$
is \emph{bounded} if each $h(x,\cdot)$
is bounded,  \emph{differentiable}
if each $h(x,\cdot)$ is differentiable,
and similarly for other analytic properties.
We write $h'(x,t)=\tfrac{d}{dt}h(x,t)$ etc.

The measure $E[\cdot]$, which was used to define
the space--time Ising measure $\mu^\b_\L$ in
Section~\ref{gr_ssec}, may be written as a
product $E=E_\times\times E_0$. 
Here $E_\times[\cdot]$ denotes a probability measure
governing the collection $D=(D_x: x\in\L)$ of 
Poisson processes conditioned to have even size,
and $E_0[\cdot]$ denotes a probability measure
governing the random variables $\xi_x\in\{0,1\}$ (for $x\in\L$).
Recall that $\s(x,t)=(-1)^{\xi_x+|D_x\cap[0,t]|}$ and that
\begin{equation*}
Z^\b_\L=
E\Big[\exp\Big(\l\sum_{x\sim y}\int_0^\b\s(x,t)\s(y,t)\,dt\Big)\Big].
\end{equation*}
Note that if $u:\L\rightarrow\RR$, then
\begin{equation}\label{L-sum_eq}
\el Lu,u\er=\frac{1}{2}\sum_{x\sim y} (u(x)-u(y))^2.
\end{equation}
Since $\s(x,t)^2=1$ for all $x\in\L$, $t\in\Ob$,
it follows that $Z^\b_\L$ is a multiple of
\begin{equation}\label{Z0_eq}
Z(0):=E\Big[\exp\Big(
-\l\int_0^\b\el L\s(\cdot,t),\s(\cdot,t)\er
\,dt\Big)\Big].
\end{equation}
In fact,
\begin{equation*}
\mu^\b_\L(f)=\frac{1}{Z(0)}
E\Big[f(\s)\exp\Big(
-\l\int_0^\b\el L\s(\cdot,t),\s(\cdot,t)\er
\,dt\Big)\Big].
\end{equation*}
The quantity $Z(0)$ 
is a special case of $Z(\bh)$ as defined in~\eqref{Zh_eq} below.
Since $\b$ and $\L$ are fixed in what follows, we will suppress them
from the notation $\mu^\b_\L$ and simply write $\mu$.
However, we will still write $Z^\b_\L$ to distinguish it from
$Z(0)$.

For a function $f:\L\times\Ob\rightarrow\RR$,
recall that the Fourier transform $\hat f$
is given by
\[
\hat f(k,l)=\sum_{x\in\L}\int_0^\b f(x,t)e^{ik\cdot x}e^{ilt}dt,
\quad k\in\tfrac{2\pi}{2N}\L,\,l\in\tfrac{2\pi}{\b}\ZZ.
\]
Throughout this section we will write
$c(x,t)$ for the Schwinger two-point function,
\[
c(x,t):=\mu\big(\s(0,0)\s(x,t)\big),
\quad x\in\L,\,t\in \Ob,
\]
given in~\eqref{tp}.
Note that $\hat c(k,l)\geq 0$
(indeed, $\hat c(k,l)=\mu[|\hat\s(k,l)|^2]/(\b|\L|)$).

The process $D$ is the set of discontinuities of $\s$.
Under $\mu$ it is absolutely continuous with respect to a Poisson
process of intensity $\d$ on $\L\times\Ob$, 
but 
the density depends
on $\L$.  The following lemma gives a `uniform
stochastic bound' on $D$.  For two point processes
$C$ and $D$ on $\L\times\Ob$ we say that 
$D$ is \emph{stochastically dominated} by $C$ if there
is a coupling $\PP$ of $C$ and $D$ such that 
$\PP(D\se C)=1$.

\begin{lemma}\label{muD_lem}
Under $\mu$, the process $D$ is stochastically
dominated by a Poisson process of intensity $2\d$.
\end{lemma}
\begin{proof}
The proof uses the space--time random-cluster 
(or \fk-) representation, which is described 
in~\cite[Chapter~2]{bjo_phd}.
(The space--time Ising measure is defined slightly
differently in~\cite{bjo_phd} than in the present work,
but the equivalence of the definitions follows from
elementary properties of Poisson processes.)
Let $\phi_{q;\l,\d}$ denote the space--time random-cluster
measure on $\L\times\Ob$.  As described 
in~\cite[Section~2.1]{bjo_phd}, a realization of $\s$ with law
$\mu$ can be obtained from a realization $\om$
with law $\phi_{2;\l,\d}$ by assigning to each connected
component spin $\pm 1$ independently with probability
$1/2$ each.  
Let $C$ denote the process of `cuts'
in $\om$.  It follows that $D\se C$.  Moreover, 
by~\cite[Corollary~2.2.13]{bjo_phd}, the process of cuts
under $\phi_{2;\l,\d}$ is
stochastically dominated by the process of cuts under 
$\phi_{1;\l/2,2\d}$.  Under the latter measure, the process
of cuts is a Poisson process with
intensity $2\d$.
\end{proof}

We will prove Theorem~\ref{irb_thm} by 
establishing a variational result, which we describe
in the next subsection.

\subsection{Gaussian domination}

For $\bh$ bounded and 
\emph{twice} differentiable we define the quantity
\begin{multline}\label{Zh_eq}
Z(\bh)=E\Big[\exp\Big(
-\l\int_0^\b
\el L[\s(\cdot,t)+h(\cdot,t)],[\s(\cdot,t)+h(\cdot,t)]\er\,dt+\\
+\frac{1}{2\d}\sum_{x\in\L}\int_0^\b h''(x,t)\s(x,t)\,dt\Big)\Big].
\end{multline}
We will deduce Theorem~\ref{irb_thm} from an
upper bound on $Z(\bh)$.  The type of bound we will derive
is similar to what is known as `Gaussian domination',
although we will not pursue any connections to Gaussian
gradient models here.  

If $\bh$ has the special property that there is 
a function $h:\Ob\rightarrow\RR$ such that $h(x,t)=h(t)$
for all $x\in\L$, then we will write $Z(h)$ for $Z(\bh)$.
Using~\eqref{L-sum_eq} we see that
\[
Z(h)=E\Big[\exp\Big(-\l\int_0^\b
\el L\s(\cdot,t),\s(\cdot,t)\er\,dt
+\frac{1}{2\d}\sum_{x\in\L}\int_0^\b h''(t)\s(x,t)\,dt
\Big)\Big].
\]
Clearly the same expression is valid if each
$h(x,\cdot)=h(\cdot)$ almost everywhere.
If, moreover, $h''(t)=0$
for all $t\in\Ob$ then it follows that
$Z(h)=Z(0)$ as given in~\eqref{Z0_eq}.

The proof of Theorem~\ref{irb_thm} rests on two
main lemmas, of which the following is the first:
\begin{lemma}\label{spatial_gd_lem}
Let $\bh$ be twice differentiable.  Then there is
some $z\in\L$
such that, writing $h(t)$ for $h(z,t)$, we have
$Z(\bh)\leq Z(h)$.
\end{lemma}
Lemma~\ref{spatial_gd_lem} will be proved in 
Section~\ref{spatial_pf_sec}.  

Next,
write $|D|:=\sum_{x\in\L}|D_x|$ for the total 
number of points in $D$, and define
\begin{equation}\label{zeta_def}
\z(r)=\mu[\cosh(r/\d)^{|D|}],\quad r\in\RR.
\end{equation}
(Recall that we write $\mu[\cdot]$ for expectation wrt $\mu$.)
It follows from Lemma~\ref{muD_lem} that $\z(r)$
is analytic in $r\in\RR$, and hence
\begin{equation}\label{2ndorder}
\z(r)=1+\frac{r^2}{2\d^2}\mu|D|+O(r^4),
\quad\mbox{as }r\rightarrow0.
\end{equation}
Note that
\begin{equation}\label{muD_eq}
\mu|D|=|\L|\mu|D_0|\leq 2\b\d|\L|,
\end{equation}
by translation invariance and Lemma~\ref{muD_lem}.

For $f:\Ob\rightarrow\RR$, let 
\[
\|f\|_2=\Big(\int_0^\b f(t)^2\,dt\Big)^{1/2}\mbox{ and }
\|f\|_\oo=\mathrm{essup}_{t\in\Ob} |f(t)|
\]
denote the $L^2$- and $L^\oo$-norms of $f$, respectively.
The second main step in proving Theorem~\ref{irb_thm}
is to establish the following result:
\begin{lemma}\label{temporal_gd_lem}
Let $h:\Ob\rightarrow\RR$ be twice
differentiable.  Then \[Z(h)\leq \z(\| h'\|_\oo)Z(0).\]
\end{lemma}

Lemma~\ref{temporal_gd_lem} will be proved
in Section~\ref{temp_pf_sec}.
Whereas Lemma~\ref{spatial_gd_lem} can be proved 
using arguments similar to those for previously
known infrared bounds, Lemma~\ref{temporal_gd_lem}
requires new ideas.
We now show how Theorem~\ref{irb_thm} follows
from Lemmas~\ref{spatial_gd_lem} 
and~\ref{temporal_gd_lem}.
\begin{proof}[Proof of Theorem~\ref{irb_thm}]
Let $\bh$ be twice differentiable and let
$h(\cdot)=h(z,\cdot)$ be as in Lemma~\ref{spatial_gd_lem}.
We will prove the result by making a particular 
choice of $\bh$, but for the time being we assume
only that there is $q\in(0,1]$ such that the set
$\{t\in\Ob:|h'(t)|\geq \|h'\|_\oo/2\}$ 
has Lebesgue measure at least $q\b$.  Then
\[
\|h'\|_2^2=\int_0^\b h'(t)^2dt\geq q\b (\|h'\|_\oo/2)^2,
\]
so by monotonicity of $\z$ we have that
\[
\z(\|h'\|_\oo)\leq
 \z\Big(\frac{2}{\sqrt{q\b}}\|h'\|_2\Big).
\]
Hence by Lemmas~\ref{spatial_gd_lem} and~\ref{temporal_gd_lem},
\[
Z(\bh)\leq Z(h)\leq \z\Big(\frac{2}{\sqrt{q\b}}\|h'\|_2\Big) Z(0).
\]
Using the fact that 
$\el L(h+\s),(h+\s)\er=2\el Lh,\s\er+\el Lh,h\er+\el L\s,\s\er$
and dividing by $Z(0)$ it follows that
\begin{multline}\label{tgd1}
\mu\Big[\exp\Big(
-2\l\int_0^\b\el Lh(\cdot,t),\s(\cdot,t)\er\,dt
+\frac{1}{2\d}\sum_{x\in\L}\int_0^\b h''(x,t)\s(x,t)\Big)\Big]\\
\leq \z\Big(\frac{2}{\sqrt{q\b}}\|h'\|_2\Big)
\exp\Big(\l\int_0^\b\el 
Lh(\cdot,t),h(\cdot,t)\er\,dt\Big).
\end{multline}
Replace $\bh$ by $\a\bh=(\a h(x,t):x\in\L,t\in\Ob)$ for
$\a>0$ and expand $\z$ and the exponentials 
in~\eqref{tgd1} as power series.
In the left-hand-side, the term of order $\a$ as $\a\rightarrow0$
is
\[
\mu\Big[-\l\int_0^\b\el Lh(\cdot,t),\s(\cdot,t)\er\,dt
+\frac{1}{2\d}\sum_{x\in\L}\int_0^\b h''(x,t)\s(x,t)\Big]=0
\]
by the $\pm$-symmetry of $\s$ under $\mu$.  Using~\eqref{2ndorder}
we find, on comparing terms of order $\a^2$, that
\begin{multline}\label{o2}
\frac{1}{2}
\mu\Big[\Big(2\l\int_0^\b\el Lh(\cdot,t),\s(\cdot,t)\er\,dt
-\frac{1}{2\d}\sum_{x\in\L}\int_0^\b h''(x,t)\s(x,t)\Big)^2\Big]\\
\leq \frac{2\mu|D|}{\d^2q\b}\int_0^\b h'(t)^2\,dt+
\l\int_0^\b\el Lh(\cdot,t),h(\cdot,t)\er\,dt.
\end{multline}

Now let $g(x,t)=a(x,t)+i b(x,t)\in\CC$, where 
$a(\cdot,\cdot),b(\cdot,\cdot):\L\times\Ob\rightarrow\RR$
are twice differentiable.
Assume each $a(x,\cdot)$ and $b(x,\cdot)$ satisfies 
the assumptions on $h$ above, with
the same $q$.  Since
\[
\begin{split}
\el Lg(\cdot,t),\s(\cdot,t)\er&=
\el La(\cdot,t),\s(\cdot,t)\er+i\el Lb(\cdot,t),\s(\cdot,t)\er,\\
\el Lg(\cdot,t),\overline{g(\cdot,t)}\er&=
\el La(\cdot,t),a(\cdot,t)\er+\el Lb(\cdot,t),b(\cdot,t)\er,\\
g''(x,t)&=a''(x,t)+i b''(x,t),\mbox{ and}\\
|g'(x,t)|^2&=a'(x,t)^2+b'(x,t)^2,
\end{split}
\]
it follows from~\eqref{o2} that
\begin{multline}\label{f1}
\mu\Big[\Big|2\l\int_0^\b\el Lg(\cdot,t),\s(\cdot,t)\er\,dt
-\frac{1}{2\d}\sum_{x\in\L}\int_0^\b g''(x,t)\s(x,t)\Big|^2\Big]\\
\leq \frac{4\mu|D|}{\d^2q\b}\int_0^\b(a'(z_1,t)^2+b'(z_2,t)^2)\,dt
+2\l\int_0^\b\el Lg(\cdot,t),\overline{g(\cdot,t)}\er\,dt,
\end{multline}
for some $z_1,z_2\in\L$.  

We now apply~\eqref{f1} with
\[
g(x,t)=e^{ik\cdot x}e^{ilt}=
\cos(k\cdot x+lt)+i\sin(k\cdot x+lt),
\quad k\in\tfrac{2\pi}{2N}\L,\,l\in\tfrac{2\pi}{\b}\ZZ.
\]
Then $g$ satisfies the assumptions above, with 
$q=2/3$.  Since $a'(z_1,t)^2+b'(z_2,t)^2\leq 2l^2$, 
the first term on the right-hand-side
in~\eqref{f1} is at most 
\[
\frac{8l^2\mu|D|}{2\d^2/3}\leq \frac{24\b|\L|}{\d}l^2.
\]
Here we used also~\eqref{muD_eq}.  Next, 
\[
\begin{split}
(Lg(\cdot,t))(x)&=\sum_{y\in\L}L(x,y)g(y,t)
=\sum_{y\in\L}L(0,y-x)e^{ik\cdot(y-x)}e^{ik\cdot x}e^{ilt}\\
&=\hat L(k)g(x,t),
\end{split}
\]
so the second term on
the right-hand-side of~\eqref{f1} is
\[
2\l\int_0^\b\el Lg(\cdot,t),\overline{g(\cdot,t)}\er\,dt
=2\l \hat L(k)\b|\L|.
\]
In the left-hand-side of~\eqref{f1} we have
\[
2\l\int_0^\b\el Lg(\cdot,t),\s(\cdot,t)\er\,dt
=2\l \hat L(k)\sum_{x\in\L}\int_0^\b\s(x,t)e^{ik\cdot x}e^{ilt}\,dt,
\]
and
\[
-\frac{1}{2\d}\sum_{x\in\L}\int_0^\b g''(x,t)\s(x,t)\,dt=
\frac{l^2}{2\d}\sum_{x\in\L}\int_0^\b \s(x,t)e^{ik\cdot x}e^{ilt}\,dt.
\]
Using the relation $|z|^2=z\bar z$ and the translation-invariance
of $\mu$, it follows that the left-hand-side of~\eqref{f1} equals
\[
\Big(2\l \hat L(k)+\frac{l^2}{2\d}\Big)^2\b|\L|\hat c(k,l).
\]
Putting it all together gives
\begin{equation}\label{messy_bound}
\hat c(k,l)\leq 
\frac{2\l\hat L(k)+48l^2/2\d}{(2\l \hat L(k)+l^2/2\d)^2}
\leq \frac{48}{2\l \hat L(k)+l^2/2\d},
\end{equation}
as required.
\end{proof}

\subsection{Proof of Lemma~\ref{spatial_gd_lem}}
\label{spatial_pf_sec}

Let $\tau:\L\rightarrow\L$ be any function.  We extend
$\tau$ to a function $\L\times\Ob\rightarrow\L\times\Ob$,
which we also denote by $\tau$, by letting $\tau(x,t)=(\tau(x),t)$.
Let $\bh=(h(x,t):x\in\L,t\in\Ob)$. 
We write $h\circ\tau$ for the usual composition of the 
functions $h$ and $\tau$, given by 
$(h\circ\tau)(x,t)=h(\tau(x,t))=h(\tau(x),t)$.  We also write
$\bh\circ\tau=((h\circ\tau)(x,t):x\in\L,t\in\Ob)$.
Note that if $\tau_1,\tau_2:\L\rightarrow\L$
then $(\bh\circ\tau_1)\circ\tau_2=\bh\circ(\tau_1\circ\tau_2)$,
by the usual associativity of function composition.

A function $\a:\L\rightarrow\L$ is an \emph{automorphism}
if, firstly, it is a bijection, and, secondly, $x\sim y$
if and only if $\a(x)\sim\a(y)$ for all $x,y\in\L$.
Since $Z(\bh)$ only depends on $\L$ through its 
connectivity structure, we see that
\begin{equation}\label{auto_eq}
Z(\bh\circ\a)=Z(\bh)\mbox{ for all automorphisms }
\a:\L\rightarrow\L.
\end{equation}
For any $y\in\L$ the map $\a:x\mapsto x+y$
is an automorphism (addition is coordinate-wise and 
interpreted modulo $2N$).
Also, any permutation of the coordinates $(x_1,\dotsc,x_d)$
is an automorphism.

The following functions 
$\rho,\rho^+,\rho^-:\L\rightarrow\L$ will be particularly
important in our proof of Lemma~\ref{spatial_gd_lem}.  
We let $\rho$ be the automorphism
of $\L$ given by:
\[
\rho:(x_1,x_2,\dotsc,x_d)\mapsto(2N-1-x_1,x_2,\dotsc,x_d).
\]
We may think of $\rho$ geometrically as 
a `reflection' in a plane parallel to,
and `just to the left of', the first coordinate plane.
Writing
\begin{equation*}
\begin{split}
\L^+&=\{(x_1,\dotsc,x_d)\in\L:0\leq x_1\leq N-1\},\mbox{ and}\\
\L^-&=\{(x_1,\dotsc,x_d)\in\L:N\leq x_1\leq 2N-1\},
\end{split}
\end{equation*}
it follows that $\rho$ bijectively maps $\L^+$
to $\L^-$ and $\L^-$ to $\L^+$.  Next we define
the functions $\rho^+,\rho^-:\L\rightarrow\L$ by
\begin{equation*}
\rho^+(x)=\left\{\begin{array}{ll}
x, & \mbox{if } x\in\L^+,\\
\rho(x), & \mbox{if } x\in\L^-,
\end{array}\right.
\quad\mbox{and}\quad
\rho^-(x)=\left\{\begin{array}{ll}
x, & \mbox{if } x\in\L^-,\\
\rho(x), & \mbox{if } x\in\L^+.
\end{array}\right.
\end{equation*}
Note that $\rho^+$ and $\rho^-$ are \emph{not} bijections,
and in particular not automorphisms.

For $\bh:\L\times\Ob\rightarrow\RR$ we define the following
two numbers:
\begin{enumerate}
\item $N(\bh)$ is the number of unordered pairs of 
adjacent elements $x\sim y$ of $\L$ such that 
$h(x,\cdot)\neq h(y,\cdot)$;
\item $N^\pm(\bh)$ is the number of pairs of 
adjacent elements $x\sim y$ such that $x\in\L^+$,
$y\in\L^-$, and  $h(x,\cdot)\neq h(y,\cdot)$.
\end{enumerate}
Equality of functions may here be interpreted pointwise
or in the almost-everywhere sense, this makes no difference
to our results.
Lemma~\ref{spatial_gd_lem} follows from the following
result:
\begin{lemma}
For any bounded, twice 
differentiable $\bh:\L\times\Ob\rightarrow\RR$
we have that 
\begin{enumerate}\label{spatial_symm_lem}
\item $Z(\bh)^2\leq Z(\bh\circ\rho^+)Z(\bh\circ\rho^-)$, and
\item $N(\bh)\geq\min\{N(\bh\circ\rho^+),N(\bh\circ\rho^-)\}$,
this inequality being strict if $N^\pm(\bh)>0$.
\end{enumerate}
\end{lemma}
Before proving Lemma~\ref{spatial_symm_lem} we show how
it implies Lemma~\ref{spatial_gd_lem}.
\begin{proof}[Proof of Lemma~\ref{spatial_gd_lem}]
The set of all functions $\tau:\L\rightarrow\L$ is finite,
so $Z(\bh\circ\tau)$ attains its maximum over $\tau$.
Let $\tau$ be chosen so that
\begin{enumerate}
\item $Z(\bh\circ\tau)$ is maximal, and
\item $N(\bh\circ\tau)$ is minimal among the maximizers
of $Z(\bh\circ\tau)$.
\end{enumerate}
Suppose $N(\bh\circ\tau)>0$.  By automorphism
invariance~\eqref{auto_eq}
we may then assume that $N^\pm(\bh\circ\tau)>0$.
It then follows from the second part of 
Lemma~\ref{spatial_symm_lem} that
\begin{equation}\label{N_ineq}
N(\bh\circ\tau)>
\min\{N(\bh\circ(\tau\circ\rho^+)),
N(\bh\circ(\tau\circ\rho^-))\}.
\end{equation}
By our choice of $\tau$, both $Z(\bh\circ(\tau\circ\rho^+))$
and $Z(\bh\circ(\tau\circ\rho^-))$ are at most equal to 
$Z(\bh\circ\tau)$, and in light also of~\eqref{N_ineq}
one of them must be strictly smaller than
$Z(\bh\circ\tau)$.  This is a contradiction, however, since
the first part of Lemma~\ref{spatial_symm_lem}
would then give $Z(\bh\circ\tau)^2<Z(\bh\circ\tau)^2$.

It follows that $N(\bh\circ\tau)=0$, which is to say that
$h(\tau(x),\cdot)=h(\tau(y),\cdot)$ for all $x\sim y$, and hence (since
$\L$ is connected) for all $x,y\in\L$.  The result follows. 
\end{proof}

To prove Lemma~\ref{spatial_symm_lem} we need one more
preliminary result.  We let $\S^+$
denote the set of functions
$\L^+\times\Ob\rightarrow\{-1,+1\}$ 
which are right continuous, and have left limits, in the second argument.
We let $\cF^+$ denote the natural
sigma-field on $\S^+$, generated by finite-dimensional
projections, as in~\eqref{proj_eq}. 
We let $\cB$ denote the Borel sigma-field on $\Ob$.
For $\s\in\S$ and $A:\S^+\rightarrow\RR$ we use the shorthand
$A(\s)$ for $A$ applied to the restriction of $\s$
to $\L^+$.  For $\tau:\L\rightarrow\L$
we write $A\circ\tau$ for the function
$\s\mapsto A(\s\circ\tau)$, where $\s\circ\tau$
is as defined in the paragraph after 
Lemma~\ref{spatial_gd_lem}.
To simplify the notation we write $A^+$
for $A\circ\rho^+$ and $A^-$ for $A\circ\rho^-$.  Note that $A^+$ and
$A^-$ have domain $\S$ rather than $\S^+$.
Recall the measure $E$ from Section~\ref{gr_ssec}.
\begin{lemma}\label{spatial_cs_lem}
Let $J$ be a finite set.  Let $A,B:\S^+\rightarrow\RR$
be bounded and $\cF^+$-measurable, and for each $j\in J$ let
$C_j,D_j:\Ob\times\S^+\rightarrow\RR$ be bounded and
$\cB\times\cF^+$-measurable.  Then
\begin{multline}\label{cs1_eq}
E\Big[\exp\Big(A^+(\s)+B^-(\s)
+\sum_{j\in J}\int_0^\b C_{j,t}^+(\s)\, 
D_{j,t}^-(\s)\,dt\Big)\Big]^2\\
\leq 
E\Big[\exp\Big(A^+(\s)+A^-(\s)
+\sum_{j\in J}\int_0^\b C_{j,t}^+(\s)\, 
C_{j,t}^-(\s)\,dt\Big)\Big]\\
\cdot E\Big[\exp\Big(B^+(\s)+B^-(\s)
+\sum_{j\in J}\int_0^\b D_{j,t}^+(\s)\, 
D_{j,t}^-(\s)\,dt\Big)\Big]
\end{multline}
\end{lemma}
\begin{proof}
We first make the following observation.  Let 
$F,G:\S^+\rightarrow\RR$ be
bounded and $\cF^+$-measurable.  Then
\begin{equation}\label{cs_eq}
\begin{split}
E[F^+(\s)G^-(\s)]^2&=E[F(\rho^+(\s))G(\rho^-(\s))]^2\\
&=E[F(\rho^+(\s))F(\rho^-(\s))]
E[G(\rho^+(\s))G(\rho^-(\s))]\\
&=E[F^+(\s)F^-(\s)]E[G^+(\s)G^-(\s)].
\end{split}
\end{equation}
This is because the sets $\L^+=\rho^+(\L^+)$
and $\L^-=\rho^-(\L^+)$ are disjoint, and the
random variables $\rho^+(\s)$ and $\rho^-(\s)$
therefore independent and identically distributed
under $E$.  For the same reason,
\begin{equation}\label{pos1_eq}
\begin{split}
E[F^+(\s)F^-(\s)]&=E[F(\s)]^2\geq0,\mbox{ and}\\
E[G^+(\s)G^-(\s)]&=E[G(\s)]^2\geq0.
\end{split}
\end{equation}

Turning now to~\eqref{cs1_eq}, the integrand in the
left-hand-side may be written as the sum
\begin{equation}\label{s1}
\sum_{k\geq0}\frac{1}{k!}
\exp\big(A^+(\s)+B^-(\s)\big)
\Big(\sum_{j\in J}\int_0^\b C_{j,t}^+(\s)\, 
D_{j,t}^-(\s)\,dt\Big)^k.
\end{equation}
Taking the expectation inside the sum
and expanding the last factor,
each summand in~\eqref{s1} may in turn be written
as a sum over $j_1,\dotsc,j_k\in J$
of a repeated integral over $t_1,\dotsc,t_k\in\Ob$
of a term of the form
\begin{equation}\label{cs3_eq}
\frac{1}{k!}E\big[e^{A^+(\s)}C^+_{j_1,t_1}(\s)\dotsb C^+_{j_k,t_k}(\s)\cdot
e^{B^-(\s)}D^-_{j_1,t_1}(\s)\dotsb D^-_{j_k,t_k}(\s)\big].
\end{equation}
The latter expectation is of the form $E[F^+(\s)G^-(\s)]$, with
\[
\begin{split}
F(\s)&=e^{A(\s)}C_{j_1,t_1}(\s)\dotsb C_{j_k,t_k}(\s),\mbox{ and}\\
G(\s)&=e^{B(\s)}D_{j_1,t_1}(\s)\dotsb D_{j_k,t_k}(\s).
\end{split}
\]
Using~\eqref{cs_eq}, and slightly abusing notation, we therefore
see that the left-hand-side of~\eqref{cs1_eq} equals
\begin{equation*}
\begin{split}
\Big(\sum_{k\geq 0}&\sum_{j_1,\dotsc,j_k\in J}
\int_0^\b dt_1\cdots\int_0^\b dt_k
\frac{1}{k!}\sqrt{E\big[F^+(\s) F^-(\s)]}
\sqrt{E\big[G^+(\s)G^-(\s)\big]}\Big)^2\\
&\leq \sum_{k\geq 0}\frac{1}{k!}\sum_{j_1,\dotsc,j_k\in J}
\int_0^\b dt_1\cdots\int_0^\b dt_k E\big[F^+(\s) F^-(\s)\big]\\
&\qquad\cdot 
\sum_{k\geq 0}\frac{1}{k!}\sum_{j_1,\dotsc,j_k\in J}
\int_0^\b dt_1\cdots\int_0^\b dt_k E\big[G^+(\s) G^-(\s)\big].
\end{split}
\end{equation*}
Here we used the Cauchy--Schwarz inequality as well 
as~\eqref{pos1_eq}.  Reversing the steps 
leading up to~\eqref{cs3_eq} for each of the two
factors gives the result.
\end{proof}

\begin{proof}[Proof of Lemma~\ref{spatial_symm_lem}]
For the first part, the aim is to write $Z(\bh)$
(see~\eqref{Zh_eq}) in terms of suitably chosen 
$A$, $B$, $C_{j,t}$ and $D_{j,t}$,
and then apply Lemma~\ref{spatial_cs_lem}.
To begin with, for simplicity of notation,
fix $t\in\Ob$ and write $\s(x)$ and $h(x)$
for $\s(x,t)$ and $h(x,t)$, respectively.
We have by~\eqref{L-sum_eq} that
\[
\el L(\s+h),\s+h\er=
\sum_{x\sim y} (\s(x)+h(x)-\s(y)-h(y))^2,
\]
which splits into the three sums
\begin{equation}\label{s1_eq}
\sum_{\substack{x\sim y\\x,y\in\L^+}} (\s(x)+h(x)-\s(y)-h(y))^2,
\end{equation}
\begin{equation}\label{s2_eq}
\sum_{\substack{x\sim y\\x,y\in\L^-}} (\s(x)+h(x)-\s(y)-h(y))^2,
\end{equation}
and
\begin{equation}\label{s3_eq}
\sum_{\substack{x\sim y\\x\in\L^+,y\in\L^-}} (\s(x)+h(x)-\s(y)-h(y))^2.
\end{equation}
Write $x\sim\L^-$ (respectively, $x\sim\L^+$)
to denote that $x$ is adjacent to some element
of $\L^-$ (respectively, $\L^+$).
Then~\eqref{s3_eq} equals
\begin{equation}\label{s4_eq}
\begin{split}
\sum_{\substack{x\in\L^+\\x\sim\L^-}} (\s(x)&+h(x))^2+
\sum_{\substack{y\in\L^-\\y\sim\L^+}} (\s(y)+h(y))^2\\
&-\sum_{\substack{x\sim y\\x\in\L^+,y\in\L^-}}
2(\s(x)+h(x))(\s(y)+h(y)).
\end{split}
\end{equation}
Since $\rho^-$ is a bijection from $\L^+$ to 
$\L^-$,~\eqref{s2_eq} equals
\begin{equation*}
\sum_{\substack{x\sim y\\x,y\in\L^+}} 
\big((\s\circ\rho^-)(x)+(h\circ\rho^-)(x)
-(\s\circ\rho^-)(y)-(h\circ\rho^-)(y)\big)^2,
\end{equation*}
and the second sum in~\eqref{s4_eq} equals
\begin{equation*}
\sum_{\substack{y\in\L^+\\y\sim\L^-}} 
\big((\s\circ\rho^-)(y)+(h\circ\rho^-)(y)\big)^2
\end{equation*}
Moreover, if $x\in\L^+$, $y\in\L^-$ with $x\sim y$, 
then $y=\rho^-(x)$.  So the last term in~\eqref{s4_eq}
equals
\begin{equation*}
\sum_{\substack{x\in\L^+\\x\sim\L^-}}
2(\s(x)+h(x))((\s\circ\rho^-)(x)+(h\circ\rho^-)(x)).
\end{equation*}

Recall that $\rho^+(x)=x$ for $x\in\L^+$.
Reintroducing $t$ into our notation, it follows that
\begin{multline}\label{L_split_eq}
-\l\int_0^\b\el 
L[\s(\cdot,t)+h(\cdot,t)],[\s(\cdot,t)+h(\cdot,t)]\er\,dt=\\
=A_1^+(\s)+B_1^-(\s)+
\sum_{\substack{x\in\L^+\\x\sim\L^-}}\int_0^\b
C_{x,t}^+(\s) D_{x,t}^-(\s)\,dt,
\end{multline}
where $A_1,B_1,C_{x,t},D_{x,t}:\S\rightarrow\RR$ are given by
\begin{multline*}
A_1(\s)=-\l\int_0^\b\Big(
\sum_{\substack{x\in\L^+\\x\sim\L^-}} (\s(x,t)+h(x,t))^2+\\
+\sum_{\substack{x\sim y\\x,y\in\L^+}} 
\big(\s(x,t)+h(x,t)-\s(y,t)-h(y,t)\big)^2
\Big)dt,
\end{multline*}
\begin{multline*}
B_1(\s)=-\l\int_0^\b\Big(
\sum_{\substack{x\in\L^+\\x\sim\L^-}} 
\big(\s(x,t)+(h\circ\rho^-)(x,t)\big)^2+\\
+\sum_{\substack{x\sim y\\x,y\in\L^+}} 
\big(\s(x,t)+(h\circ\rho^-)(x,t)-\s(y,t)
-(h\circ\rho^-)(y,t)\big)^2
\Big)dt,
\end{multline*}
and
\begin{equation}\label{CD_eq}
\begin{split}
C_{x,t}(\s)&=\sqrt{2\l}\big(\s(x,t)+h(x,t)\big),\\
D_{x,t}(\s)&=\sqrt{2\l}\big(\s(x,t)+(h\circ\rho^-)(x,t)\big).
\end{split}
\end{equation}

Next, 
\begin{equation}\label{int_split_eq}
\frac{1}{2\d}\sum_{x\in\L}\int_0^\b h''(x,t)\s(x,t)\,dt
=A_2^+(\s)+B_2^-(\s),
\end{equation}
where 
\begin{equation*}
A_2(\s)=\frac{1}{2\d}\sum_{x\in\L^+}\int_0^\b
h''(x,t)\s(x,t)\,dt
\end{equation*}
and
\begin{equation*}
B_2(\s)=\frac{1}{2\d}\sum_{x\in\L^+}\int_0^\b
(h\circ\rho^-)''(x,t)\s(x,t)\,dt.
\end{equation*}
Thus, from~\eqref{Zh_eq},
$Z(\bh)$ is of the form of the left-hand-side
of~\eqref{cs1_eq}, with $A(\s)=A_1(\s)+A_2(\s)$,
$B(\s)=B_1(\s)+B_2(\s)$, $J=\{x\in\L^+:x\sim\L^-\}$,
and $C_{x,t}$ and $D_{x,t}$ as in~\eqref{CD_eq}.

Reversing the steps leading to~\eqref{L_split_eq}
and~\eqref{int_split_eq} 
(and recalling that $\rho^+$ is a bijection from
$\L^-$ to $\L^+$) shows that 
\[
Z(\bh\circ\rho^+)=E\Big[\exp\Big(A^+(\s)+A^-(\s)
+\sum_{j\in J}\int_0^\b C_{j,t}^+(\s)\, 
C_{j,t}^-(\s)\,dt\Big)\Big]
\]
and that
\[
Z(\bh\circ\rho^-)=E\Big[\exp\Big(B^+(\s)+B^-(\s)
+\sum_{j\in J}\int_0^\b D_{j,t}^+(\s)\, 
D_{j,t}^-(\s)\,dt\Big)\Big].
\]
The first part of the lemma now follows
from Lemma~\ref{spatial_cs_lem}.

For the second part define the numbers
\begin{itemize}
\item $N^+(\bh)$ as the number of unordered pairs of 
adjacent elements $x\sim y$ of $\L^+$ such that 
$h(x,\cdot)\neq h(y,\cdot)$;  and
\item $N^-(\bh)$ as the number of unordered pairs of 
adjacent elements $x\sim y$ of $\L^-$ such that 
$h(x,\cdot)\neq h(y,\cdot)$.
\end{itemize}
Then $N(\bh)=N^+(\bh)+N^-(\bh)+N^\pm(\bh)$,
whereas $N(\bh\circ\rho^+)=2N^+(\bh)$ and
$N(\bh\circ\rho^-)=2N^-(\bh)$.  The result follows
immediately from this observation.
\end{proof}

\subsection{Proof of Lemma~\ref{temporal_gd_lem}}
\label{temp_pf_sec}

Write $D_x=\{t_1^x,t_2^x,\dotsc,t_{|D_x|}^x\}$
for the points of $D_x$ ordered so that
$0<t_1^x<t_2^x<\cdots<t_{|D_x|}^x<\b$.  For
convenience we also write $t_{|D_x|+1}^x=t^x_1$.
Note that if $h$ is twice differentiable, then
\begin{equation}\label{hsigma}
\begin{split}
\int_0^\b h''(t)\s(x,t)dt&=
\sum_{j=1}^{|D_x|}\int_{t^x_j}^{t^x_{j+1}}h''(t)\s(x,t^x_j)dt\\
&=\sum_{j=1}^{|D_x|} \s(x,t^x_j)(h'(t^x_{j+1})-h'(t^x_j))\\
&=-2\sum_{j=1}^{|D_x|}h'(t^x_j)\s(x,t^x_j).
\end{split}
\end{equation}
Here we used the fact that $\s(x,t_j^x)=-\s(x,t_{j+1}^x)$.
Moreover, since $\s(x,t)=(-1)^{\xi_x+|D_x\cap[0,t]|}$
we have that $\s(x,t^x_j)=(-1)^{\xi_x+j}$.  Hence
\begin{equation}\label{pf}
Z(h)=E\Big[\exp\Big(
-\l\int_0^\b\el L\s(\cdot,t),\s(\cdot,t)\er\,dt-
\frac{1}{\d}
\sum_{x\in\L}\sum_{j=1}^{|D_x|}h'(t_j^x)(-1)^{\xi_x+j}\Big)\Big].
\end{equation}
Note that the form~\eqref{pf} does not require $h$
to be twice differentiable.  We say that $h$ is 
\emph{weakly differentiable}, and that
$h'$ is a \emph{weak derivative} of $h$,
if there is a constant $c$ such that
\[
h(t)=\int_0^t h'(s)\,ds+c\quad\mbox{for all }0\leq t<\b.
\]
Since weak derivatives are defined up to a set of zero
measure,~\eqref{pf} is well-defined if we take $h'$ to be
any weak derivative of $h$.  In this section we will let
$Z(h)$ denote the quantity in~\eqref{pf}, and the standing
assumption on $h$ will be that it is weakly differentiable
with a bounded weak derivative.  Also note that 
$Z(h)=Z(h+c)$ for any constant $c$, so we may occasionally
assume that $h(0)=0$.

By the monotonicity of $\z$,
Lemma~\ref{temporal_gd_lem} will be proved if we show:
\begin{equation}\label{gd_lem}
\mbox{there is } 0\leq r\leq\|h'\|_\oo \mbox{ such that } 
Z(h)/Z(0)\leq \z(r).
\end{equation}
The proof of~\eqref{gd_lem}
is preceded by a number of preliminary results, of which
the following general fact is the first.
For each $n\geq1$ let 
\[
\cO^{(n)}=\bigcup_{k=0}^{2^{n-1}-1}[(2k+1)2^{-n},(2k+2) 2^{-n})
\]
be the union of those level $n$ dyadic subintervals 
of $[0,1)$ whose left endpoints are odd multiples
of $2^{-n}$.  Write $B_n(t)=\one\{t\in\cO^{(n)}\}$
for $t\in[0,1]$.
\begin{lemma}\label{binary_lem}
Let $m\geq1$ and let $T=(T_1,\dotsc,T_m)\in[0,1]^m$
be a random vector with square integrable density 
$p:[0,1]^m\rightarrow[0,\oo)$. 
Then for each $b\in\{0,1\}^m$,
\[
P\big((B_n(T_1),\dotsc,B_n(T_m))=b\big)
\rightarrow\frac{1}{2^m},
\quad\mbox{as }n\rightarrow\oo.
\]
\end{lemma}
Intuitively, Lemma~\ref{binary_lem}
states that the level $n$ binary digits of the $T_j$
are asymptotically independent and uniform
as $n\rightarrow\oo$.
\begin{proof}
For each $n\geq1$ and $A\se\{1,\dotsc,m\}$, let 
\[
R_n^A(t)=\prod_{j\in A} (-1)^{B_n(t_j)},\quad t\in[0,1]^m,
\]
denote the Rademacher function.  Then 
$(R_n^A:A\neq\es,n\geq1)$ are orthonormal in the
Hilbert space $L^2([0,1]^m)$.  Writing
\[
\hat p(A,n)=E(R^A_n(T))=
\int_0^1\dotsb\int_0^1 p(t_1,\dotsc,t_m)R^A_n(t_1,\dotsc,t_m)
dt_1\dotsc dt_m,
\]
it follows (from Bessel's inequality or otherwise) that
for each $\es\neq A\se\{1,\dotsc,m\}$ we have 
$\sum_{n\geq1}\hat p(A,n)^2<\oo$.  In particular,
$\hat p(A,n)\rightarrow0$ as $n\rightarrow\oo$.
For any $A\se\{1,\dotsc,m\}$,
\begin{multline*}
\one\{B_n(T_j)=1\;\forall j\in A,\;
B_n(T_j)=0\;\forall j\not\in A\}=\\
=2^{-m}\prod_{j\in A} \big(1-(-1)^{B_n(T_j)}\big)
\prod_{j\not\in A} \big(1+(-1)^{B_n(T_j)}\big).
\end{multline*}
Expanding the products on the right, we obtain a sum of
terms of the form $\pm R^C_n(T)$ for $C\se\{1,\dotsc,m\}$.  
The term $R^\es_n(T)=1$ appears exactly once.  Taking expected
value and letting $n\rightarrow\oo$ gives the result.
\end{proof}

The following technical lemma will enable us to
apply Lemma~\ref{binary_lem} to the process $D$.
\begin{lemma}\label{density_lem}
For each $x\in\L$, let $m_x\geq0$ be an integer, and
for $j\in\{1,\dotsc,2m_x\}$ let 
$I^x_j=[a^x_j,b^x_j]\se[0,\b)$ be intervals
such that $0<a^x_1<b^x_1\leq a^x_2<b^x_2\leq\dotsb
\leq a^x_{2m_x}<b^x_{2m_x}<\b$.
Let $A$ denote the event that: for each $x\in\L$
we have $|D_x|=2m_x$, and for all 
$j\in\{1,\dotsc,2m_x\}$ we have $t^x_j\in I^x_j$.
Then the law of  $(t^x_j:x\in\L,j\in\{1,\dotsc,2m_x\})$
under $\mu(\cdot\mid A)$
has a square-integrable density with respect to Lebesgue
measure on
$I:=\prod_{x\in\L}\prod_{j=1}^{2m_x} I^x_j$.
\end{lemma}
\begin{proof}
Recall that $E=E_0\times E_\times$, where $E_0$
governs $\xi$ and $E_\times$ governs $D$.
The law of $D$ under $\mu$ has density 
\begin{equation}\label{q_eq}
q(D)=E_0\Big[\frac{1}{Z_\L}\exp\Big(
\l\sum_{x\sim y}\int_0^\b\s(x,t)\s(y,t)\,dt\Big)\Big]
\end{equation}
with respect to $E_\times$.
From standard properties of conditional expectation
it follows that the law of $D$ under $\mu(\,\cdot\mid A)$
has density \[
\tilde q(D)=\frac{q(D)}{E_\times[q(D)\mid A]}
\] 
with respect to $E_\times(\,\cdot\mid A)$.
The law of $(t^x_j:x\in\L,j\in\{1,\dotsc,2m_x\})$
under $E_\times(\,\cdot\mid A)$ is the uniform 
distribution on $I$, by standard properties of
Poisson processes.
It is clear from~\eqref{q_eq} that $\tilde q$ is square
integrable.
\end{proof}

Let $r\in\RR$
be a real number and $n\geq1$ an integer, and let
\[
W'_{r,n}(t)=r(-1)^{\lfloor 2^nt/\b\rfloor}.
\]
Thus $W'_{r,n}(t)$ takes the two values $\pm r$, and changes sign
at the level $n$ dyadics, ie points of the form $t=k2^{-n}\b$
for $k\in\{0,1,\dotsc,2^n-1\}$.
Let $W_{r,n}$ be the antiderivative of $W'_{r,n}$, given by
\[
W_{r,n}(t)=\int_0^t W'_{r,n}(s)\,ds,\qquad\mbox{for all }0\leq t<\b.
\]
See Figure~\ref{white_fig}.
\begin{figure}[hbt]
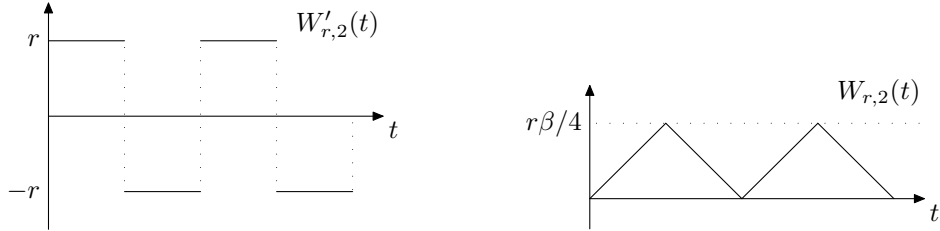

\centering
\includegraphics{irb.2}\qquad\qquad
\includegraphics{irb.3}
\caption{The functions $W'_{r,n}(t)$ (left)
and $W_{r,n}(t)$ (right) for $n=2$.}
\label{white_fig}   
\end{figure}
Recall that $\z(r)=\mu[\cosh(r/\d)^{|D|}]$.

\begin{lemma}\label{white_lem}
We have that
\[
\lim_{n\rightarrow\oo}Z(W_{r,n})=\z(r) Z(0).
\]
\end{lemma}

\begin{proof}
From~\eqref{pf} and rotational invariance we need to show that
\begin{equation}\label{e1}
\frac{Z(W_{r,n})}{Z(0)}=
\mu\Big[\exp\Big(\frac{r}{\d}\sum_{x\in\L}\sum_{j=1}^{|D_x|}
(-1)^{\xi_x+j+\lfloor 2^nt^x_j/\b\rfloor}\Big)\Big]
\rightarrow \z(r)
\end{equation}
as $n\rightarrow\oo$.  
To motivate the argument that follows, let 
$(Y_j^x:x\in\L,j\geq1)$
be independent random variables,
each taking the values $\pm 1$ with equal probability
under $\mu$.  Then (by conditioning on the $|D_x|$)
\[
\mu\Big[\exp\Big(\frac{r}{\d}\sum_{x\in\L}\sum_{j=1}^{|D_x|}
Y_j^x\Big)\Big]=
\mu\big[\mu[\exp(rY^0_1/\d)]^{|D|}\big]=
\mu[\cosh(r/\d)^{|D|}]=\z(r).
\]
The strategy for proving~\eqref{e1} will be to first
condition on the `rough' locations of the $t^x_j$,
and then use Lemma~\ref{binary_lem} to deduce that
the conditional joint distribution of the numbers
$(-1)^{\xi_x+j+\lfloor 2^nt^x_j/\b\rfloor}$ 
approaches that of
the $Y^x_j$.  Here are the details.

Writing
\[
H_n=\frac{1}{\d}\sum_{x\in\L}\sum_{j=1}^{|D_x|}
(-1)^{\xi_x+j+\lfloor 2^nt^x_j/\b\rfloor},
\]
first note that the sequence $(e^{rH_n}:n\geq1)$
is uniformly integrable under $\mu$, for each $r\in\RR$.
Indeed, a sufficient condition for uniform 
integrability is that
\[
\sup_{n\geq1}\mu\big[|e^{rH_n}|^2\big]
=\sup_{n\geq1}\mu[e^{2rH_n}]<\oo.
\]
This follows from Lemma~\ref{muD_lem} and the fact that
$|H_n|\leq |D|/\d$.

Let $\eps>0$ be arbitrary.
Let $M$ be some (large) integer, and let $A_M$
denote the event that each $|D_x|$ is at most $2M$.
Then $\mu(A_M)\rightarrow1$ as $M\rightarrow\oo$.
Next, let $L$ be another (large)
integer, and let $B_L$ denote the event that for each 
$x\in\L$ there are integers 
$0\leq k_1^x<k_2^x<\dotsb<k_{|D_x|}^x\leq 2^L-1$
such that each 
$t^x_j\in[k^x_j2^{-L}\b,(k_j^x+1)2^{-L}\b)$.
Then also $\mu(B_L)\rightarrow1$ as $L\rightarrow\oo$.
For each $\a>0$ we have that
\[
\sup_{n\geq1}\mu[e^{r H_n}\one_{A_M^\tc\cup B_L^\tc}]\leq
\sup_{n\geq1}\mu[e^{r H_n}\one\{e^{rH_n}>\a\}]+
\a \mu(A_M^\tc\cup B_L^\tc).
\]
By uniform integrability, this can be made smaller than
$\eps$ by first choosing $\a$ large enough that
$\sup_{n\geq1}\mu[e^{r H_n}\one\{e^{rH_n}>\a\}]<\eps/2$
and then $M,L$ large enough that 
$\a \mu(A_M^\tc\cup B_L^\tc)<\eps/2$.
In what follows we 
assume that $M$ and $L$ are fixed, and large enough
that
\begin{equation}\label{e2}
\sup_{n\geq1}\mu[e^{r H_n}\one_{A_M^\tc\cup B_L^\tc}]<\eps.
\end{equation}

Let $\mu'$ denote $\mu$ conditioned on the following:
\begin{enumerate}
\item that $A_M$ and $B_L$ both occur;
\item the vector $\xi=(\xi_x:x\in\L)$;
\item the sizes $|D_x|=2m_x$ for all $x\in\L$;
\item and the numbers 
$0\leq k_1^x<k_2^x<\dotsb<k_{2m_x}^x\leq 2^L-1$
such that each
$t^x_j\in[k^x_j2^{-L}\b,(k_j^x+1)2^{-L}\b)$.
\end{enumerate}
Let $m=\sum_{x\in\L}m_x$ so that $|D|=2m$, and let
\[
T^x_j=\frac{t^x_j-k^x_j2^{-L}\b}{2^{-L}\b}.
\]
Then $T=(T^x_j:x\in\L,j\in\{1,\dotsc,2m_x\})$
is a random vector in $[0,1]^{2m}$.  Moreover, by
Lemma~\ref{density_lem}
the law of $T$ under $\mu'$ has square integrable density
with respect to Lebesgue measure on $[0,1]^{2m}$.

Let $n\geq 1$ and
write $X^x_j=\one\{T^x_j\in\cO^{(L+n)}\}$.  Note that
$X^x_j$ has the same parity as 
$\lfloor 2^{L+n}t^x_j/\b\rfloor$.  Let 
$a=(a^x_j:x\in\L,1\leq j\leq 2m_x)\in\{-1,1\}^{2m}$
be arbitrary, and write
\[
X'=((-1)^{\xi_x+j+X^x_j}:x\in\L,j\in\{1,\dotsc,2m_x\}).
\]
By Lemma~\ref{binary_lem} we have that
$\mu'(X'=a)\rightarrow 2^{-m}$ as $n\rightarrow\oo$
for any $a\in\{-1,+1\}^{2m}$, and hence that
\begin{equation}\label{e3}
\mu'\Big[\exp\Big(\frac{r}{\d}\sum_{x\in\L}\sum_{j=1}^{2m_x}
(-1)^{\xi_x+j+X^x_j}\Big)\Big]\rightarrow
\mu'\Big[\exp\Big(\frac{r}{\d}\sum_{x\in\L}\sum_{j=1}^{2m_x}
Y_j^x\Big)\Big]=\cosh(r/\d)^{2m}.
\end{equation}
Now, with some slight abuse of notation for 
conditional expectation,
\[
\begin{split}
\frac{Z(W_{r,L+n})}{Z(0)}&=
\mu\Big[\exp\Big(\frac{r}{\d}\sum_{x\in\L}\sum_{j=1}^{|D_x|}
(-1)^{\xi_x+j+X^x_j}\Big)\Big]\\
&=\mu[e^{rH_{L+n}}\one_{A_M^\tc\cup B_L^\tc}]
+\mu\Big[\mu'\Big[\exp\Big(\frac{r}{\d}\sum_{x\in\L}\sum_{j=1}^{2m_x}
(-1)^{\xi_x+j+X^x_j}\Big)\Big]\Big].
\end{split}
\]
The first term is smaller than $\eps$, by~\eqref{e2}.
For fixed $M$ and $L$,
the outer expectation in the second term is over a finite
set of possibilities, namely the possible values of the
$\xi_x$, the $|D_x|$ (each being at most $2M$) and 
the $k^x_j$, as listed in the definition of $\mu'$.
By making $n$ sufficiently large, we may
by~\eqref{e3} assume that 
the integrand differs from $\cosh(r)^{2m}$ by at most
$\eps$ for each such possibility.  It then follows that
\[
\Big|\frac{Z(W_{r,L+n})}{Z(0)}-\mu[\cosh(r/\d)^{|D|}]\Big|< 2\eps.
\]
The result follows.
\end{proof}

The rough strategy for proving~\eqref{gd_lem}
will be to define a procedure by which
$h$ can be altered so that it more and more resembles
$W_{r,n}$ for some $r$, whilst increasing $Z(h)$.
First we need a Cauchy--Schwarz-type inequality
along the lines of Lemma~\ref{spatial_cs_lem}.

Let $\theta:\Ob\rightarrow\Ob$ be
given by $\theta(t)=\b-t$.  We may think of $\theta$
as a reflection of the circle 
$\Ob=\{e^{2\pi it/\b}:t\in[0,\b)\}$ in the
real line.  Recall that the measure
$E_\times[\cdot]$ governs $D$ only, which is a
Poisson process conditioned on each $D_x$ having even size.
For each $x\in\L$, let $D^+_x=D_x\cap(0,\b/2)$ and 
$D^-_x=D\cap(\b/2,\b)$.  Write $D^+=(D_x^+:x\in\L)$
and $D^-=(D_x^-:x\in\L)$;  also write
$\theta D^-=\{\theta t:t\in D^-\}\se(0,\b/2)$.

\begin{lemma}\label{cs_vert_lem}
The measure $E_\times[\cdot]$ is `reflection positive' in that
for any bounded, measurable function $F$ of $D^+$ we have that
\begin{equation}\label{po_rp_eq}
E_\times[F(D^+)F(\theta D^-)]\geq 0.
\end{equation}
Consequently we have for all bounded measurable $F$, $G$ that
\begin{equation}\label{cs2_eq}
E_\times[F(D^+)G(\theta D^-)]^2\leq 
E_\times[F(D^+)F(\theta D^-)]E_\times[G(D^+)G(\theta D^-)].
\end{equation}
\end{lemma}
\begin{proof}
For~\eqref{po_rp_eq}, condition on the parity of each $|D^+_x|$.
Note that $|D^-_x|$ necessarily has the same parity as $|D^+_x|$.
Given this parity, $\theta D^-_x$ is independent of,
and identically distributed as, $D^+_x$.  This gives
\[
E_\times[F(D^+)F(\theta D^-)]=
\sum_{j\in\{0,1\}^\L}E_\times[F(D^+)\mid 
\forall x\in\L,\,|D_x|\equiv j_x\mbox{ mod } 2]^2,
\]
and hence~\eqref{po_rp_eq}.

The Cauchy--Schwarz inequality~\eqref{cs2_eq} is a standard 
consequence of~\eqref{po_rp_eq} seeing as $D^+$ and $\theta D^-$
are identically distributed:
for any $t\in\RR$ we have that
\[
\begin{split}
0&\leq E_\times[(F(D^+)+t G(D^+))
(F(\theta D^-)+t G(\theta D^-))]\\
&=E_\times[F(D^+)F(\theta D^-)]+
2t E_\times[F(D^+)G(\theta D^-)]+
t^2E_\times[G(D^+)G(\theta D^-)].
\end{split}
\]
So the discriminant
\[
4 E_\times[F(D^+)G(\theta D^-)]^2-
4E_\times[F(D^+)F(\theta D^-)]E_\times[G(D^+)G(\theta D^-)]\leq 0,
\]
which gives~\eqref{cs2_eq}.
\end{proof}

For any  $f:\Ob\rightarrow\RR$
define $f_+$ and $f_-$ by
\begin{equation}\label{ominous1}
f_+(t)=\left\{
\begin{array}{ll}
f(t), & \mbox{if } t\in[0,\b/2),\\
f(\theta t), & \mbox{if } t\in[\b/2,\b),
\end{array}\right.
\end{equation}
and
\begin{equation}\label{ominous2}
f_-(t)=\left\{
\begin{array}{ll}
f(\theta t), & \mbox{if } t\in[0,\b/2),\\
f(t) & \mbox{if } t\in[\b/2,\b).
\end{array}\right.
\end{equation}
If $f$ has a weak derivative $f'$ then
$f_+$ and $f_-$ have weak derivatives $f'_+$
and $f'_-$ satisfying
\begin{equation}\label{plus}
f'_+(t)=\left\{
\begin{array}{ll}
f'(t), & \mbox{if } t\in(0,\b/2),\\
-f'(\theta t), & \mbox{if } t\in(\b/2,\b),
\end{array}\right.
\end{equation}
and
\begin{equation}\label{minus}
f'_-(t)=\left\{
\begin{array}{ll}
-f'(\theta t), & \mbox{if } t\in(0,\b/2),\\
f'(t) & \mbox{if } t\in(\b/2,\b).
\end{array}\right.
\end{equation}
See Figure~\ref{pm_fig} for an illustration.
\begin{figure}[hbt]
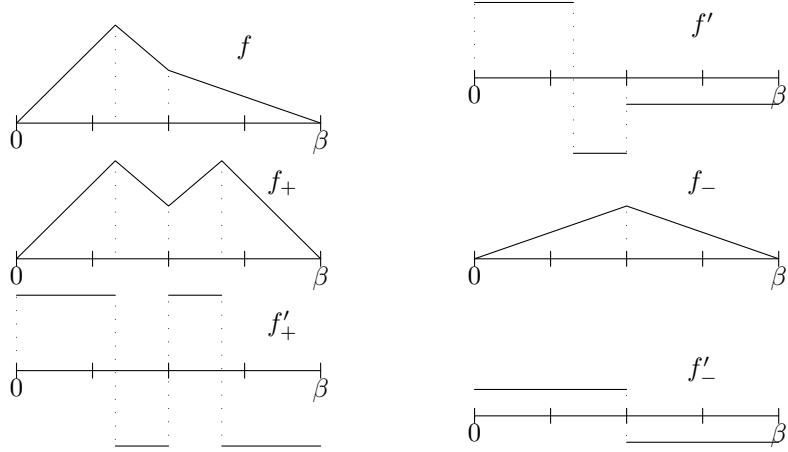

\centering
\includegraphics{irb.4}
\qquad\qquad\includegraphics{irb.5}\\
\includegraphics{irb.6}
\qquad\qquad\includegraphics{irb.7}\\
\includegraphics{irb.8}
\qquad\qquad\includegraphics{irb.9}
\caption{An example of $f_+$ and $f_-$ and their derivatives.}
\label{pm_fig}   
\end{figure}

The following result parallels Lemma~\ref{spatial_symm_lem}.
\begin{lemma}\label{vert_pm_lem}
Let $h:\Ob\rightarrow\RR$ have a bounded weak
derivative.  Then $Z(h)\leq \max\{Z(h_+), Z(h_-)\}$.
\end{lemma}

\begin{proof}
Conditioning on $\xi=(\xi_x:x\in\L)\in\{0,1\}^\L$, 
we may write $Z(h)=\frac{1}{2^{|\L|}}\sum_{\xi}Z(h\mid\xi)$,
where 
\begin{multline}\label{pf5}
Z(h\mid\xi):=E_\times\Big[\exp\Big(
-\l\int_0^\b 
\el L\s(\cdot,t),\s(\cdot,t)\er\,dt-
\\-\frac{1}{\d}\sum_{x\in\L}\sum_{j=1}^{|D_x|}(-1)^{j+\xi_x}h'(t^x_j)
\Big)\Big].
\end{multline}
We will express $Z(h\mid\xi)$ in the form
\begin{equation}\label{c2}
Z(h\mid\xi)=E_\times[\exp\big(A_\xi(D^+)+B_\xi(\theta D^-)\big)]
\end{equation}
 and then use
Lemma~\ref{cs_vert_lem}.  First we need some notation.

Write $r_j^x$, $j=1,\dotsc,|D^+_x|$, for the elements of $D^+_x$
ordered so that $r_j^x<r_{j+1}^x$ for all $j$.
Also let $r_0^x=0$ and $r_{|D_x^+|+1}^x=\b/2$.  Similarly, write
$s_j^x$, $j=1,\dotsc,|D^-_x|$, for the elements of $D^-_x$
\emph{but ordered so that $s_j^x>s_{j+1}^x$ for all $j$};
in other words so that
$\theta s_j^x<\theta s_{j+1}^x$ for all $j$.
Also let $s_0^x=0$ and $s_{|D_x^-|+1}^x=\b/2$.  
Note that
\begin{equation*}
t^x_j=\left\{\begin{array}{ll}
r^x_j, & \mbox{for } 1\leq j\leq |D_x^+|,\\
s^x_{|D_x|-j+1}, & \mbox{for } |D_x^+|+1\leq j\leq |D_x|.
\end{array}\right.
\end{equation*}
See Figure~\ref{circle_fig}.
\begin{figure}[hbt]
\centering
\includegraphics{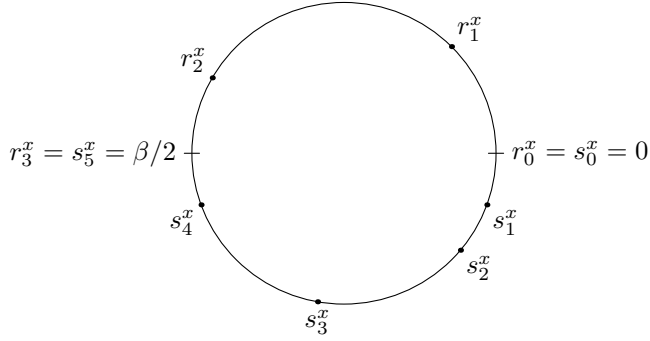}
\caption{Ordering of the elements of 
$D^+_x=\{r_1^x,r_2^x\}$ and $D^-_x=\{s_1^x,s_2^x,s_3^x,s_4^x\}$.}
\label{circle_fig}   
\end{figure}

In the last sum in~\eqref{pf5} we have
\[
\begin{split}
\sum_{j=1}^{|D_x|}(-1)^{j+\xi_x}h'(t_j^x)&=
\sum_{j=1}^{|D^+_x|}(-1)^{j+\xi_x}h'(r_j^x)
+\sum_{j=1}^{|D^-_x|}(-1)^{|D_x|-j+1+\xi_x}h'(s_j^x)\\
&=\sum_{j=1}^{|D^+_x|}(-1)^{j+\xi_x}h'_+(r_j^x)
+\sum_{j=1}^{|\theta D^-_x|}(-1)^{j+\xi_x}h'_-(\theta s_j^x).
\end{split}
\]
Hence~\eqref{c2} holds with
\[
A_\xi(D^+)=-\l\int_0^{\b/2}
\el L\s(\cdot,t),\s(\cdot,t)\er\,dt
-\frac{1}{\d}\sum_{x\in\L}\sum_{j=1}^{|D^+_x|}(-1)^{j+\xi_x}h'_+(r^x_j)
\]
and
\[
B_\xi(D^+)=-\l\int_0^{\b/2}
\el L\s(\cdot,t),\s(\cdot,t)\er\,dt
-\frac{1}{\d}\sum_{x\in\L}\sum_{j=1}^{|D^+_x|}(-1)^{j+\xi_x}h'_-(r^x_j).
\]
Next note that
\[
\begin{split}
A_\xi(\theta D^{-})&=\sum_{j=1}^{|\theta D_x^-|}(-1)^{j+\xi_x}h'_+(\theta s_j^x)
=\sum_{j=1}^{|D_x^-|}(-1)^{j+\xi_x}(-h'_+(s_j^x))\\
&=\sum_{j=1}^{|D^-_x|}(-1)^{j+\xi_x+1}h'_+(t_{|D_x|-j+1}^x)
=\sum_{j=|D_x^+|+1}^{|D_x|}(-1)^{j+\xi_x}h'_+(t_{j}^x).
\end{split}
\]
Here we used the fact that $|D_x^+|+|D_x^-|=|D_x|$ is even
so that $(-1)^{|D_x^+|}=(-1)^{|D_x^-|}$.
It follows that 
$Z(h_+\mid\xi)=
E_\times\big[\exp\big(A_\xi(D^+)+A_\xi(\theta D^-)\big)\big]$.
Similarly 
$Z(h_-\mid\xi)=
E_\times\big[\exp\big(B_\xi(D^+)+B_\xi(\theta D^-)\big)\big]$.
Therefore we get from Lemma~\ref{cs_vert_lem} that
\begin{equation}\label{c3}
\begin{split}
Z(h\mid\xi)^2&\leq 
E_\times\big[\exp\big(A_\xi(D^+)+A_\xi(\theta D^-)\big)\big]\\
&\qquad \cdot E_\times\big[\exp\big(B_\xi(D^+)+B_\xi(\theta D^-)\big)\big]\\
&=Z(h_+\mid\xi)Z(h_-\mid\xi).
\end{split}
\end{equation}
From~\eqref{c3} and the usual Cauchy--Schwarz inequality,
\begin{equation}\label{c4}
Z(h)^2\leq\Big(\frac{1}{2^{|\L|}}\sum_{\xi\in\{0,1\}}
\sqrt{Z(h_+\mid\xi)}
\sqrt{Z(h_-\mid\xi)}\Big)^2\leq
Z(h_+)Z(h_-).
\end{equation}
Finally,~\eqref{c4} implies that at least one of 
$Z(h_+)$ and $Z(h_-)$ is at least $Z(h)$,
as required.
\end{proof}

\begin{definition}[Symmetrization]
Let $h:\Ob\rightarrow\RR$ have a bounded weak derivative
and let $t\in[0,\b/2)$.  The
\emph{symmetrization of $h$ at $t$}
is the function $g$ given by the completing the
following steps:
\begin{enumerate}
\item Let $\tilde h:s\mapsto h(s+t)$;
\item Let $\tilde h_+$ and $\tilde h_-$
be as in~\eqref{ominous1} and~\eqref{ominous2};
\item Let
\[
\tilde g=\left\{\begin{array}{ll}
\tilde h_+, &\mbox{ if } 
Z(\tilde h_+)\geq Z(\tilde h_+),\\
\tilde h_-, &\mbox{ otherwise};
\end{array}\right.
\]
\item Let $g(s)=\tilde g(s-t)$.
\end{enumerate}
\end{definition}
Note that if $g$ is the symmetrization of $h$
at $t$ then $g$ is symmetric at $t$ and $t+\b/2$;
also $Z(g)\geq Z(h)$ by Lemma~\ref{vert_pm_lem}.

The strategy for proving~\eqref{gd_lem} is 
to successively symmetrize $h$ at
dyadic points of finer and finer partition.
Thereby our function
more and more resembles $W_{r,n}$ for some $r$.  By 
Lemma~\ref{white_lem} we know that 
$\lim_{n\rightarrow\oo}Z(W_{r,n})=\z(r) Z(0)$.
We now make this precise.

\begin{definition}[Snippet]\label{snippet_def}
Let $f:\Ob\rightarrow\RR$.  For $n\geq0$ we call
a function $g:\Ob\rightarrow\RR$ a 
\emph{level $n$ snippet} of $f$ if 
\begin{enumerate}
\item there is $k\in\{0,1,\dotsc,2^n-1\}$ and $a\in\{0,1\}$ 
such that
$g(t)=f(k2^{-n}\b+(-1)^a t)$ for all $0<t<2^{-n}\b$, and
\item $g(m2^{-n}\b+t)=g(m2^{-n}\b-t)$ for all
$m\in\{0,1,\dotsc,2^n-1\}$ and all $0<t<2^{-n}\b$.
\end{enumerate}
\end{definition}
Thus a level $n$ snippet of $f$ 
repeats the values that $f$  takes
on an interval $(k2^{-n}\b,(k+1)2^{-n}\b)$, but alternates
between `the right way' and `backwards'.
Note that $h_+$ and $h_-$ are
level 1 snippets of $h$.

\begin{lemma}\label{snippet_lem}
Let $h:\Ob\rightarrow\RR$ have a bounded weak derivative.
There is a sequence $(h_n:n\geq 0)$ of functions 
$h_n:\Ob\rightarrow\RR$ such that
\begin{enumerate}
\item $h_0=h$,
\item for each $n\geq1$, $h_{n}$ is a level $n$ snippet
of $h_{n-1}$, and
\item for each $n\geq1$, $Z(h_n)\geq Z(h_{n-1})$.
\end{enumerate}
\end{lemma}
\begin{proof}
We are free to set $h_0=h$.  Let $h_1$ be 
the symmetrization of $h_0$ at $0$,
that is to say $h_1$ is either
$h_+$ or $h_-$, chosen so that $Z(h_1)\geq Z(h)$.
This is possible due to
Lemma~\ref{vert_pm_lem}.
Next let $h_2$ be the symmetrization of $h_1$
at $\b/4$.  Then $h_2$ is a level 2 snippet of $h_1$,
and Lemma~\ref{vert_pm_lem} implies that 
$Z(h_2)\geq Z(h_1)$.  
To get $h_3$ from $h_2$, the symmetrization procedure 
must be carried out twice, as follows.
First let $g$ be the symmetrization of 
$h_2$ at $\b/8$;
then let $h_3$ be the symmetrization of $g$
at $3\b/8$.  Lemma~\ref{vert_pm_lem} implies that 
$Z(h_3)\geq Z(h_2)$.

This process is carried out inductively.  To get 
$h_n$ from $h_{n-1}$,  symmetrization
must be carried out $(2^n-2^{n-1})/2=2^{n-2}$
times, once at each point of $(0,\b/2)$ which is dyadic
of level $n$ but not of level $n-1$.  This gives a level
$n$ snippet $h_n$ of $h_{n-1}$ such that 
$Z(h_n)\geq Z(h_{n-1})$.
\end{proof}

For  $h:\Ob\rightarrow\RR$ weakly differentiable, write
$\|h\|'=\|h'\|_\oo$
The following lemma is immediate from the definition~\eqref{pf}.
\begin{lemma}\label{cty_lem}
$Z(\cdot)$ is continuous in the norm $\|\cdot\|'$.
\end{lemma}

We can now complete the proof of~\eqref{gd_lem}, and
hence Lemma~\ref{temporal_gd_lem}.

\begin{proof}[Proof of~\eqref{gd_lem}]
Let $(h_n:n\geq0)$ be the sequence produced
by Lemma~\ref{snippet_lem}.  Note that $h_n$ is
a level $n$ snippet of $h$ itself.  In particular
$\|h'_n\|_\oo\leq\|h'\|_\oo=:M$ for all $n\geq1$.
By symmetry we may assume that 
$h_n(t)=h(k_n 2^{-n}\b+t)$ for all 
$0\leq t<2^{-n}\b$ and some $k_n\in\{0,1,\dotsc,2^n-1\}$;
that is, $a=0$ in Definition~\ref{snippet_def}. 
Moreover, since $h_{n+1}$ is a level $n$ snippet of 
$h_n$ we may also assume that $k_{n+1}\in\{2k_n,2k_n+1\}$
for all $n$;  that is, the restriction of
$h_{n+1}$ to $[0,2^{-(n+1)}\b)$ equals the restriction
of $h_n$ to either $[0,2^{-(n+1)}\b)$
or to $[2^{-(n+1)}\b,2^{-n}\b)$.

Let $t_n=(k_n+1/2)2^{-n}\b$ be the midpoint.  
Clearly the sequence
$t_n$ is convergent, with limit $t_*$ say.
Let $r_n=h_n'(2^{-(n+1)}\b)=h'(t_n)$.
By continuity of $h'$ we have that 
$r_n\rightarrow r:=h'(t_*)\in[-M,M]$.
In fact, since $\Ob$ is compact, $h'$ is uniformly
continuous.  So for $\eps>0$ given we have for
large enough $n$ that $|h'(s)-h'(t)|<\eps$
whenever $|s-t|\leq2^{-n}\b$.  This implies that 
$\|h_n'-W'_{r_n,n}\|_\oo<\eps$ for large enough $n$.
Since $\|W'_{r,n}-W'_{r_n,n}\|_\oo\rightarrow0$
as $n\rightarrow\oo$ it follows from
Lemma~\ref{cty_lem} that for any $\eps'>0$
we have
\[
|Z(h_n)-Z(W_{r,n})|\leq
|Z(h_n)-Z(W_{r_n,n})|+|Z(W_{r_n,n})-Z(W_{r,n})|<\eps'
\]
whenever $n$ is large enough.
Hence by Lemma~\ref{snippet_lem},
\[
Z(h)\leq Z(h_n)< Z(W_{r,n})+\eps'
\]
for $n$ large enough.  Since $\eps'>0$ was arbitrary 
it follows from Lemma~\ref{white_lem}
that $Z(h)\leq\z(r) Z(0)$.  Since
$\z$ is an even function we may assume that $r\geq0$.
\end{proof}


\section{Mean-field behaviour of the susceptibility}
\label{mf_sec}

The main objective of this section is to prove Theorem~\ref{mf_thm},
giving the critical exponent value $\g=1$.
The arguments in this section are inspired by similar
arguments for the classical Ising model in~\cite{aiz82,af,ag}.
Apart from Theorem~\ref{irb_thm}, the main component in the proof is a
pair of new differential inequalities for the susceptibility.  The
proof of these inequalities uses the random-parity representation
of~\cite{bjogr}, which we briefly describe next 
(see also~\cite{crawford_ioffe} for the closely related
random-current representation).

\subsection{The random-parity representation}

Throughout this subsection and the next, $\L$ and $\b$ will be fixed 
and finite.
Write $\L\times\Ob=K$.  The random-parity 
representation allows one to write, 
for each finite set $A\se K$,
\begin{equation}\label{rp1}
\mu^\b_\L\Big(\prod_{(x,t)\in A}\s(x,t)\Big)=
\frac{E(\partial\psi^A)}{E(\partial\psi^\es)},
\end{equation}
where $\psi^A$ is a certain random
labelling of $K$ using the labels `even' and
`odd' with `source set' $A$, and $\partial\psi^A$ is a positive weight
associated with the labelling.  The $\psi^A$ are constructed using
Poisson processes, but the measure $E$ is \emph{not} the same
as in Section~\ref{gr_ssec}.  
Throughout Section~\ref{mf_sec} we will use $E$ for the random-parity
measure only.
 Elements of $K$ will simply
be denoted by $x,y,a,b,\ldots$, and we let
\begin{equation*}
F=\{xy: x=(u,t)\in K, y=(v,t)\in K
\mbox{ for some } u\sim v\in\L\}
\end{equation*}
be the set of unordered pairs of `adjacent' elements of $K$.
The point $(0,0)\in K$ will be denoted by $0$.

The labelling $\psi^A$ is constructed using a Poisson process $S$ on
$F$, of intensity $\l$.  
As one traverses each `circle' $\{u\}\times\Ob$ in $K$, the
label alternates between `even' and `odd' in such a way that the label
always changes at (i) points $x\in A$, and (ii) points $x$ and $y$
such that $xy\in S$.  Moreover, these are the only types of points
where the label is allowed to change (this imposes constraints on
$S$).  See Figure~\ref{rpr_fig} for an illustration of such labellings.
\begin{figure}[hbtp]
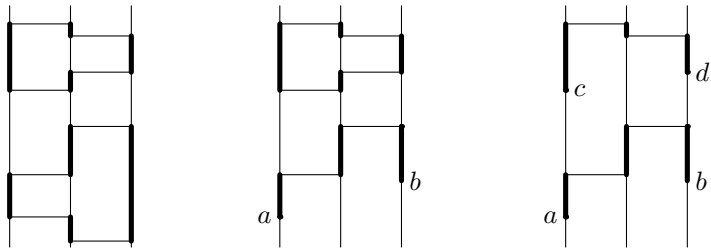

\includegraphics{irb.12}
\hspace{1.3cm} 
\includegraphics{irb.13} 
\hspace{1.3cm} 
\includegraphics{irb.14} 
\caption[Colourings in the random-parity representation]
{Labellings $\psi^A$ for $A=\es$ (left),
$A=\{a,b\}$ (middle) and $A=\{a,b,c,d\}$ (right).
Horizontal line segment represent elements of $S$.
Thick vertical
segments are `odd' and thin segments `even'.}
\label{rpr_fig}
\end{figure}
One may see that in such a labelling $\psi^A$ the `odd' subset of $K$
forms a collection of geometric `paths' between elements of $A$,
together with a collection of `loops'.  A labelling as described above
can only be defined if $A$ has even size;  if $A$ has odd size we define
the weight $\partial\psi^A=0$ (this is consistent with~\eqref{rp1}).

Although there is a natural notion of connectivity along `odd' paths,
it turns out to be much more fruitful to consider a more complex
notion of connectivity in
\emph{triples} $(\psi_1^A,\psi_2^B,\D)$.  Here $\psi_1^A$ and $\psi_2^B$ are
independent labellings and $\D$ is an independent 
Poisson process of `cuts' of intensity $4\d$.
This is described in detail in~\cite{bjogr};  here is a brief account.
Let $S_1$ and $S_2$ denote the (independent) Poisson processes on $F$
used to construct $\psi_1^A$ and $\psi_2^B$, respectively.  Elements 
$xy\in S_1\cup S_2$ may be interpreted as `bridges' which connect the
points $x$ and $y$.  Connections may traverse such bridges, and may
also traverse subintervals of $K$ \emph{except that} connections are
blocked at points $x\in\D$ such that both  $\psi_1^A$ and $\psi_2^B$
are `even' at $x$ (`odd' labels thus `cancel' $\D$).
For $a,b\in K$ we write $\{a\lra b\}$ for the event that
$a$ and $b$ are connected in the triple $(\psi_1^A,\psi_2^B,\D)$.

There are two main technical tools in the random-parity
representation. 
\begin{enumerate}
\item
\emph{The switching lemma}~\cite[Theorem~4.2]{bjogr}
implies that for any $a,b\in K$ and any two finite sets 
$A,B\se K$,
\[
E(\partial\psi_1^A\partial\psi_2^B\one\{a\lra b\})=
E(\partial\psi_1^{A\sd ab}\partial\psi_2^{B\sd ab}\one\{a\lra b\}).
\]
Here $A\sd ab$ is short-hand for the set-theoretic symmetric difference
$A\sd\{a,b\}$.  
The main manifestation of the switching lemma is the identity
\begin{equation*}
E(\partial\psi_1^A\partial\psi_2^{ab})=
E(\partial\psi_1^{A\sd ab}\partial\psi_2^\es\one\{a\lra b\}).
\end{equation*}
(The connection $a\lra b$ is automatic in the left-hand-side since
there is an odd path in $\psi_2^{ab}$.)
\item\emph{Conditioning on clusters.}  It is often useful to consider
  how a triple $(\psi_1^A,\psi_2^B,\D)$ interacts with a third
(independent) labelling $\psi_3^C$.  One may then condition on the set
$C_{1,2}(x)$ of points connected to some given point $x\in K$ in 
$(\psi_1^A,\psi_2^B,\D)$, and consider the restrictions of 
$\psi_1^A$, $\psi_2^B$ and $\psi_3^C$ to $C_{1,2}(x)$ and
$K\sm C_{1,2}(x)$ `separately'.  The details of this procedure are
technical and depend on the precise situation in which it is to be
used (care must be taken to get the correct `sources' in the
restricted labellings).  Rather than attempting to describe 
the details we point 
to~\cite[Lemma~4.6 and (5.6)--(5.8)]{bjogr}, 
where applications of this method are described in detail.
\end{enumerate}

For simplicity of notation we will, in this subsection and the next, 
write $\el\s_A\er$ for the quantity in~\eqref{rp1},
and will write $Z=E(\partial\psi^\es)$.
We write $\el\s_A;\s_B\er$ for
$\el\s_A\s_B\er-\el\s_B\er\el\s_A\er=\el\s_{A\sd B}\er-\el\s_B\er\el\s_A\er$.

Here is an example of the random-parity representation in 
action.  Let
\begin{equation*}
\chi_\L=\chi_\L(\d,\l,\b):=
\sum_{x\in\L}\int_0^\b \mu^\b_\L(\s(0,0)\s(x,t))\,dt
\end{equation*}
denote the finite-volume, positive-temperature approximation
of the susceptibility~\eqref{sus_eq}.  We see (using the
expression in Definition~\ref{st_def}) that
\begin{equation}\label{l1}
\frac{\partial\chi_\L}{\partial\l}=\int_Kdx\int_Fd(yz)
\el\s_0\s_x;\s_y\s_z\er.
\end{equation}
Using the switching lemma,
\begin{equation}\label{rp2}
\begin{split}
\el\s_0\s_x;\s_y\s_z\er&=
\frac{1}{Z}E(\partial\psi_1^{0xyz})-
\frac{1}{Z^2}E(\partial\psi_1^{0x}\partial\psi_2^{yz})\\
&=\frac{1}{Z^2}
E(\partial\psi_1^{0xyz}\partial\psi_2^\es)-
\frac{1}{Z^2}
E(\partial\psi_1^{0xyz}\partial\psi_2^\es\one\{y\lra z\})\\
&=\frac{1}{Z^2}
E(\partial\psi_1^{0xyz}\partial\psi_2^\es\one\{y\not\lra z\}).
\end{split}
\end{equation}
In particular $\frac{\partial\chi_\L}{\partial\l}\geq0$.

\subsection{Differential inequalities}

Let
\begin{equation*}
B_\L=B_\L(\d,\l,\b):=
\sum_{x\in\L}\int_0^\b\mu^\b_\L(\s(0,0)\s(x,t))^2\,dt
\end{equation*}
be the finite-volume, positive-temperature approximation of 
the bubble-diagram~\eqref{b2}.
In addition to Theorem~\ref{irb_thm}, the main step in proving
Theorem~\ref{mf_thm} is to establish the following two differential
inequalities:
\begin{lemma}\label{diff_lem}
We have that
\begin{equation}\label{di1}
4d\chi_\L^2\geq\frac{\partial\chi_\L}{\partial\l}\geq
4d\chi_\L^2-4dB_\L\chi_\L-2d\l B_\L \frac{\partial\chi_\L}{\partial\l}
-8d\d B_\L \Big(-\frac{\partial\chi_\L}{\partial\d}\Big)
\end{equation}
and
\begin{equation}\label{di2}
2\chi_\L^2\geq-\frac{\partial\chi_\L}{\partial\d}\geq
2\chi_\L^2-2B_\L\chi_\L-\l B_\L \frac{\partial\chi_\L}{\partial\l}
-4\d B_\L \Big(-\frac{\partial\chi_\L}{\partial\d}\Big).
\end{equation}
\end{lemma}
These inequalities are analogous to inequalities
for the classical Ising model in~\cite{aiz82}, 
and the proof follows a similar outline.
We recall from~\eqref{rp2} that
$\frac{\partial\chi_\L}{\partial\l}\geq 0$, and remark that
$\frac{\partial\chi_\L}{\partial\d}\leq 0$ (see~\eqref{rp3} below).

\begin{proof}
We start with~\eqref{di1}.  From~\eqref{l1} we have that 
\begin{equation}\label{di3}
\begin{split}
\frac{\partial\chi_\L}{\partial\l}&=\int_Kdx\int_Fd(yz)
[\el\s_0\s_y\er\el\s_x\s_z\er+\el\s_0\s_z\er\el\s_x\s_y\er+
U_4(0,x,y,z)]\\
&=4d\chi_\L^2+\int_Kdx\int_Fd(yz) U_4(0,x,y,z),
\end{split}
\end{equation}
where 
\begin{equation*}
U_4(a,b,c,d)=\el\s_a\s_b\s_c\s_d\er
-\el\s_a\s_b\er\el\s_c\s_d\er
-\el\s_a\s_c\er\el\s_b\s_d\er
-\el\s_a\s_d\er\el\s_b\s_c\er
\end{equation*}
is sometimes called the `fourth Ursell function'.  
Note that $U_4$ is symmetric in its four arguments.  We have that
$U_4\leq0$.  In fact
\begin{multline*}
E(\partial\psi_1^{abcd}\partial\psi_2^\es)
-E(\partial\psi_1^{ab}\partial\psi_2^{cd})
-E(\partial\psi_1^{ac}\partial\psi_2^{bd})
-E(\partial\psi_1^{ad}\partial\psi_2^{bc})\\
=E(\partial\psi_1^{abcd}\partial\psi_2^\es
[1-\one\{c\lra d\}-\one\{b\lra d\}
-\one\{b\lra c\}]),
\end{multline*}
and the quantity in square
brackets is either 0 or $-2$, the latter occurring
if and only if all four points $a,b,c,d$ are connected.  Applying the
switching lemma we arrive at the identity
\begin{equation*}
U_4(a,b,c,d)=-2\frac{1}{Z^2}E(\partial\psi_1^{ab}\partial\psi_2^{cd}
\one\{a\lra c\}).
\end{equation*}
The upper bound in~\eqref{di1} follows.

The lower bound in~\eqref{di1} will be obtained by
bounding
\[
\frac{1}{Z^2}E(\partial\psi_1^{ab}\partial\psi_2^{cd}\one\{a\lra c\})
=\frac{1}{2}|U_4(a,b,c,d)|
\]
from above and using~\eqref{di3}.  Let $\psi_3^{cd}$
be an independent labelling.  In the configuration
$\psi_3^{cd}$ there is an
odd path $\xi_3^{cd}$ from $c$ to $d$, which is called the
`backbone'  of the configuration (see~\cite[Section~3.3]{bjogr}).
Let $C_{1,2}(a)$ denote the connected cluster of $a$ in the triple 
$(\psi_1,\psi_2,\D)$.  (All connectivites in this proof will
refer to this triple.)
Conditioning on the cluster $C_{1,2}(a)$ 
as in~\cite[(5.6)--(5.8)]{bjogr} we find that
\begin{multline}\label{bb0}
E(\partial\psi_1^{ab}\partial\psi_2^\es\partial\psi_3^{cd}
\one\{\xi_3^{cd}\cap C_{1,2}(a)=\es\})\\
\leq Z E(\partial\psi_1^{ab}\partial\psi_2^\es
\el\s_c\s_d\er_{K\sm C_{1,2}(a)}),
\end{multline}
where $\el\s_c\s_d\er_{K\sm C_{1,2}(a)}$
denotes the correlation~\eqref{rp1} in the
smaller region $K\sm C_{1,2}(a)$.
Also as in~\cite{bjogr}, we further find that
\[
E(\partial\psi_1^{ab}\partial\psi_2^\es
\el\s_c\s_d\er_{K\sm C_{1,2}(a)})
=E(\partial\psi_1^{ab}\partial\psi_2^{cd}
\one\{a\not\lra c\}).
\]
Thus 
\begin{equation}\label{bb1}
Z E(\partial\psi_1^{ab}\partial\psi_2^{cd}
\one\{a\lra c\})\leq
E(\partial\psi_1^{ab}\partial\psi_2^\es\partial\psi_3^{cd}
\one\{\xi_3^{cd}\cap C_{1,2}(a)\neq\es\}).
\end{equation}

The rest of the proof of the lower bound in~\eqref{di1} will be based
on bounding the right-hand side of~\eqref{bb1}.
There are two main cases to consider in~\eqref{bb1},
namely whether or not $c\in C_{1,2}(a)$.  In case $c\in C_{1,2}(a)$
we get
\begin{equation*}
\begin{split}
E(\partial\psi_1^{ab}\partial\psi_2^\es\partial\psi_3^{cd}
\one\{c\in C_{1,2}(a)\})&=
Z\el\s_c\s_d\er E(\partial\psi_1^{ab}\partial\psi_2^\es
\one\{a\lra c\})\\
&=Z^3\el\s_c\s_d\er\el\s_b\s_c\er\el\s_a\s_c\er.
\end{split}
\end{equation*}

The case $c\not\in C_{1,2}(a)$ splits into two subcases, because the
first point $u$ on $\xi_3^{cd}$ in $C_{1,2}(a)$ 
is then either (i) an endpoint of a bridge
on $\xi_3^{cd}$ whose other endpoint $v$ is not in $C_{1,2}(a)$,
or (ii) a point of $\D$ on the boundary of $C_{1,2}(a)$.  
See Figure~\ref{backbone_fig}.
\begin{figure}[hbt]
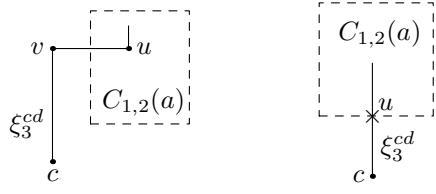

\centering
\includegraphics{irb.10}\qquad\qquad
\includegraphics{irb.11}
\caption{Cases for the first point $u$ on
$\xi_3^{cd}$ in $C_{1,2}(a)$.  Left: (i) $u$ is the endpoint
of a bridge.  Right: (ii) $u$ is at a `cut' (point of $\D$).}
\label{backbone_fig}   
\end{figure}

Let us first
consider the case when $u$ is the endpoint of a bridge.  Then
$\xi_3^{cd}$ decomposes as $\zeta\circ\zeta'$, where
$\zeta:c\rightarrow v$ and $\zeta':u\rightarrow d$.  Moreover,
$\zeta\cap C_{1,2}(a)=\es$.  We may therefore argue as
in~\eqref{bb0}--\eqref{bb1} for $\zeta$.
Using also~\cite[Lemma~3.3]{bjogr} and the
{\gks}-inequality (proved for the present model 
in~\cite[Lemma~2.2.20]{bjo_phd}),
we may therefore in this case
bound the right-hand-side of~\eqref{bb1} above by
\begin{equation}\label{bb2}
\begin{split}
&\l\int_Fd(uv)\el\s_u\s_d\er E(\partial\psi_1^{ab}
\partial\psi_2^\es\el\s_c\s_v\er_{K\sm C_{1,2}(a)}
\one\{u\in C_{1,2}(a)\})\\
&=\l Z\int_Fd(uv) \el\s_u\s_d\er
E(\partial\psi_1^{ab}\partial\psi_2^{cv}
\one\{c\not\in C_{1,2}(a), u\in C_{1,2}(a)\}).
\end{split}
\end{equation}
By the switching lemma, the latter expectation equals
\begin{equation*}
E(\partial\psi_1^{bu}\partial\psi_2^{aucv}
\one\{c\not\in C_{1,2}(a), u\in C_{1,2}(a)\}).
\end{equation*}
By conditioning on the cluster $C_{1,2}(a)$ and using the
{\gks}-inequality, this is at most
$\el\s_a\s_u\er E(\partial\psi_1^{bu}\psi_2^{cv}
\one\{u\not\lra v\})$.

Let us now consider the case when the first point
$u$ on $\xi_3^{cd}$ in $C_{1,2}(a)$ is a point of $\D$ on the
boundary of $C_{1,2}(a)$.  Then $\xi_3^{cd}$ decomposes as
$\zeta\circ\zeta'$, where $\zeta:c\rightarrow u$, 
$\zeta':u\rightarrow d$ and $\zeta\cap C_{1,2}(a)=\{u\}$.  If $u$ is
removed from $\D$ then $C_{1,2}(a)$ is enlarged, and what was
previously $C_{1,2}(a)$ becomes $C_{1,2}^u(a)$, the set of points
which can be reached from $a$ \emph{without passing $u$}.  We may
argue as for~\eqref{bb1} and~\eqref{bb2} again to see that
the right-hand-side of~\eqref{bb1} is 
in this case at most
\begin{equation*}
\begin{split}
&4\d\int_Kdu \el\s_u\s_d\er
E(\partial\psi_1^{ab}\partial\psi_2^\es
\el\s_c\s_u\er_{K\sm C_{1,2}^u(a)}
\one\{b\in C_{1,2}^u(a),  u\in C_{1,2}(a)\})\\
&=4\d Z
\int_Kdu \el\s_u\s_d\er E(\partial\psi_1^{ab}\partial\psi_2^{cu}
\one\{c\not\in C^u_{1,2}(a), b\in C_{1,2}^u(a), u\in C_{1,2}(a)\}).
\end{split}
\end{equation*}
By the switching lemma,
the latter expectation is at most
\begin{equation*}
\begin{split}
E(\partial\psi_1^{bu}\psi_2^{ac}
\one\{c\not\in C^u_{1,2}(a), b\in C_{1,2}^u(a)\})&=
E(\partial\psi_1^{bc}\psi_2^{au}
\one\{c\not\in C^u_{1,2}(a), b\in C_{1,2}^u(a)\})\\
&\leq\el\s_a\s_u\er E(\partial\psi_1^{bc}\partial\psi_2^\es
\one\{b\overset{u}{\lra} c\}),
\end{split}
\end{equation*}
where we have conditioned on the cluster $C_{1,2}^u(a)$ and used the
\gks-inequality for the upper bound. 

So far we have established that
\begin{equation}\label{u1}
\begin{split}
(0\leq) -\frac{1}{2}U_4(a,b,c,d)\leq& 
\el\s_a\s_c\er \el\s_b\s_c\er \el\s_c\s_d\er\\
&+\l\int_Fd(uv) \el\s_a\s_u\er \el\s_d\s_u\er
E(\partial\psi_1^{bu}\partial\psi_2^{cv} \one\{u\not\lra v\})\\
&+4\d\int_Kdu \el\s_a\s_u\er \el\s_d\s_u\er
E(\partial\psi_1^{bc}\partial\psi_2^\es
\one\{b\overset{u}{\lra} c\}).
\end{split}
\end{equation}
Whereas $U_4(a,b,c,d)$ is symmetric in
$a$, $b$, $c$, $d$ the right-hand-side 
of~\eqref{u1} is
not.  Averaging with respect to the transposition $b\lra c$ we
arrive at the upper bound
\begin{multline}
-U_4(a,b,c,d)\leq
\el\s_a\s_c\er \el\s_b\s_c\er \el\s_c\s_d\er+
\el\s_a\s_b\er \el\s_b\s_c\er \el\s_b\s_d\er\\
+\l\int_Fd(uv) \el\s_a\s_u\er \el\s_d\s_u\er
\big[ E(\partial\psi_1^{bu}\partial\psi_2^{cv} \one\{u\not\lra v\})
+E(\partial\psi_1^{cu}\partial\psi_2^{bv} \one\{u\not\lra v\})\big]\\
+8\d\int_Kdu \el\s_a\s_u\er \el\s_d\s_u\er
E(\partial\psi_1^{bc}\partial\psi_2^\es
\one\{b\overset{u}{\lra} c\}).
\end{multline}
Thus, setting $a=y,b=x,c=0,d=z$, 
it follows that the quantity
\begin{equation*}
-\int_Kdx\int_Fd(yz) U_4(0,x,y,z)
\end{equation*}
which appears in~\eqref{di3} is at most
\begin{multline}\label{l5}
\chi_\L\int_Fd(yz) \el\s_0\s_y\er\el\s_0\s_z\er
+\int_Kdx\el\s_0\s_x\er\int_Fd(yz) \el\s_x\s_y\er\el\s_x\s_z\er\\
+\l\int_Kdx\int_Fd(uv) \big[ E(\partial\psi_1^{xu}
\partial\psi_2^{0v} \one\{u\not\lra v\})
+E(\partial\psi_1^{0u}\partial\psi_2^{xv} \one\{u\not\lra v\})\big]
\int_Fd(yz)\el\s_y\s_u\er \el\s_z\s_u\er\\
+8\d\int_Kdx\int_Kdu E(\partial\psi_1^{0x}\partial\psi_2^\es
\one\{0\overset{u}{\lra} x\})\int_Fd(yz)
\el\s_y\s_u\er \el\s_z\s_u\er.
\end{multline}
The Cauchy--Schwarz inequality 
implies that
\begin{equation*}
\int_Fd(yz) \el\s_0\s_y\er\el\s_0\s_z\er\leq 2dB_\L;
\end{equation*}
using this together with translation invariance 
and~\eqref{rp2} shows that the
quantity in~\eqref{l5} is at most
\begin{equation*}
4d\chi_\L B_\L+2d\l B_\L\frac{\partial\chi_\L}{\partial\l}
+8d\d B_\L\Big(-\frac{\partial\chi_\L}{\partial\d}\Big).
\end{equation*}
Together with~\eqref{di3}, this proves the lower bound in~\eqref{di1}.

For the upper bound in~\eqref{di2} we note that
\begin{equation}\label{rp3}
-\frac{\partial\chi_\L}{\partial\d}=
2\int_Kdx\int_Kdy\frac{1}{Z^2}
E(\partial\psi_1^{0x}\partial\psi_2^\es\one\{0\overset{y}{\lra}x\})
\end{equation}
where the notation $0\overset{y}{\lra}x$ signifies that \emph{if}
a path connects $0$ and $x$, then it must contain $y$
(see~\cite[Theorem~4.10]{bjogr}).  By the
switching lemma
\begin{equation*}
\frac{1}{Z^2}E(\partial\psi_1^{0x}\partial\psi_2^\es\one\{0\overset{y}{\lra}x\})
=\frac{1}{Z^2}E(\partial\psi_1^{0y}\partial\psi_2^{yx}\one\{0\overset{y}{\lra}x\})
\leq\el\s_0\s_y\er\el\s_y\s_x\er.
\end{equation*}
Together with~\eqref{rp3} and translation invariance,
this proves the upper bound in~\eqref{di2}.

The lower bound in~\eqref{di2} is similar in spirit
to the lower bound  in~\eqref{di1}, but differs in
the details.  We start by recalling that
\begin{equation}\label{de1}
\begin{split}
-\frac{\partial\chi_\L}{\partial\d}&=2\int_Kdx\int_Kdy
\frac{1}{Z^2}E(\partial\psi_1^{0x}\partial\psi_2^\es
\one\{0\overset{y}{\lra}x\})\\
&=2\chi_\L^2-\frac{2}{Z^2}\int_Kdx\int_Kdy
E(\partial\psi_1^{0x}\partial\psi_2^\es
\one\{0\overset{y}{\lra}x\}^c).
\end{split}
\end{equation}
In words, the event $\{0\overset{y}{\lra}x\}^c$ is that there is some
path from $0$ to $x$ which avoids $y$.  By the switching lemma and the
method for~\eqref{bb1} we have that 
\begin{equation*}
\begin{split}
\frac{1}{Z^3}E(\partial\psi_1^{0y}&\partial\psi_2^\es
\partial\psi_3^{xy}\one\{\xi_3^{xy}\cap C_{1,2}^y(0)=\{y\}\})\\
&\leq
\frac{1}{Z^2}E(\partial\psi_1^{0y}\partial\psi_2^\es
\el\s_x\s_y\er_{K\sm C_{1,2}^y(0)})\\
&=\frac{1}{Z^2}E(\partial\psi_1^{0y}\partial\psi_2^{xy}
\one\{0\overset{y}{\lra}x\})
=\frac{1}{Z^2}E(\partial\psi_1^{0x}\partial\psi_2^\es
\one\{0\overset{y}{\lra}x\}).
\end{split}
\end{equation*}
Note that certainly $0\lra y$ in $(\psi_1^{0y},\psi_2^\es,\D)$, since
$\psi_1$ has sources $0,y$;  so in this situation the complement of
the event $\{\xi_3^{xy}\cap C_{1,2}^y(0)=\{y\}\}$ is the event 
$\{\xi_3^{xy}\cap C_{1,2}^y(0)\supsetneq\{y\}\}$
that $\{y\}$ is a strict subset of $\xi_3^{xy}\cap C_{1,2}^y(0)$.  
Thus
\begin{equation}\label{l6}
\frac{1}{Z^2}E(\partial\psi_1^{0x}\partial\psi_2^\es
\one\{0\overset{y}{\lra}x\}^c)
\leq\frac{1}{Z^3}E(\partial\psi_1^{0y}\partial\psi_2^\es
\partial\psi_3^{xy}\one\{\xi_3^{xy}\cap C_{1,2}^y(0)
\supsetneq\{y\}\}).
\end{equation}

We consider the cases whether or not $x\in C_{1,2}(0)$ in the
expectation in the right-hand-side of~\eqref{l6};
note that $C^y_{1,2}(0)\se C_{1,2}(0)$.  The case 
$x\in C_{1,2}(0)$ gives at most
\begin{equation*}
E(\partial\psi_1^{0y}\partial\psi_2^\es\partial\psi_3^{xy}
\one\{x\in C_{1,2}(0)\})=Z^3\el\s_x\s_y\er^2\el\s_0\s_x\er.
\end{equation*}
Again, the case $x\not\in C_{1,2}(0)$ decomposes into the subcases
when the first point $u$ on $\xi_3^{xy}$ which lies in $C_{1,2}(0)$ is
(i) the endpoint of a bridge whose other endpoint $v$ does not lie in
$C_{1,2}(0)$, or (ii) a cut in $\D$ on the boundary of $C_{1,2}(0)$.

In case (i), the backbone $\xi_3^{xy}$ decomposes as
$\zeta\circ\zeta'$ where $\zeta:x\rightarrow v$, 
$\zeta':u\rightarrow y$ and $\zeta\cap C_{1,2}(0)=\es$.  
As for~\eqref{bb2} it follows that the expectation in
the right-hand-side of~\eqref{l6} is in this case at most
\begin{multline}
\l Z\int_F d(uv)\el\s_u\s_y\er 
E(\partial\psi_1^{0y}\partial\psi_2^{xv}
\one\{u\lra 0, v\not\lra 0\})\\
=\l Z\int_F d(uv)\el\s_u\s_y\er 
E(\partial\psi_1^{uy}\partial\psi_2^{0uxv}
\one\{u\not\lra v, u\not\lra x\})\\
\leq \l Z\int_F d(uv)\el\s_u\s_y\er^2 
E(\partial\psi_1^{xv}\partial\psi_2^{0u}\one\{u\not\lra v\}).
\end{multline}

In case (ii) the backbone $\xi_3^{xy}$ decomposes as
$\zeta\circ\zeta'$, where $\zeta:x\rightarrow u$,
$\zeta':u\rightarrow y$ and $\zeta\cap C_{1,2}(0)=\{u\}$.
The expectation in
the right-hand-side of~\eqref{l6} is in this case at most
\begin{multline}
4\d Z\int_Kdu\el\s_u\s_y\er E(\partial\psi_1^{0y}\partial\psi_2^{xu}
\one\{0\lra u, 0\overset{u}{\lra} x\})\\
\leq 4\d Z\int_Kdu\el\s_u\s_y\er^2 E(\partial\psi_1^{0x}\partial\psi_2^\es
\one\{0\overset{u}{\lra} x\}).
\end{multline}

So far we have showed that
\begin{multline}\label{l7}
\frac{1}{Z^2}E(\partial\psi_1^{0x}\partial\psi_2^\es
\one\{0\overset{y}{\lra}x\}^c)\\
\leq \el\s_x\s_y\er^2\el\s_0\s_x\er+
\l \int_F d(uv)\el\s_u\s_y\er^2 \frac{1}{Z^2}
E(\partial\psi_1^{xv}\partial\psi_2^{0u}\one\{u\not\lra v\})\\
+4\d \int_Kdu\el\s_u\s_y\er^2 \frac{1}{Z^2}
E(\partial\psi_1^{0x}\partial\psi_2^\es\one\{0\overset{u}{\lra} x\}).
\end{multline}
Whereas the left-hand-side in~\eqref{l7} is symmetric under the
transposition $0\lra x$, the right-hand-side is not.  Averaging with
respect to this transposition we see that we may replace the
right-hand-side in~\eqref{l7} by
\begin{multline}
\frac{1}{2}\big(\el\s_x\s_y\er^2\el\s_0\s_x\er+
\el\s_0\s_y\er^2\el\s_0\s_x\er\big)\\
+\frac{\l}{2} \int_F d(uv)\el\s_u\s_y\er^2 \frac{1}{Z^2}
\big[E(\partial\psi_1^{xv}\partial\psi_2^{0u}\one\{u\not\lra v\})+
E(\partial\psi_1^{0v}\partial\psi_2^{xu}\one\{u\not\lra v\})\big]\\
+4\d \int_Kdu\el\s_u\s_y\er^2 \frac{1}{Z^2}
E(\partial\psi_1^{0x}\partial\psi_2^\es\one\{0\overset{u}{\lra} x\}).
\end{multline}
It follows that the integral 
\begin{equation*}
2\int_Kdx\int_Kdy \frac{1}{Z^2}E(\partial\psi_1^{0x}\partial\psi_2^\es
\one\{0\overset{y}{\lra}x\}^c)
\end{equation*}
which appears in~\eqref{de1} is at most
\begin{equation*}
2 B_\L \chi_\L + \l B_\L\frac{\partial\chi_\L}{\partial\l}
+4\d B_\L\Big(-\frac{\partial\chi_\L}{\partial\d}\Big).
\end{equation*}
This proves the lower bound in~\eqref{di2}.
\end{proof}

\subsection{The bubble-diagram}

This section is devoted to bounds on the
bubble-diagram~\eqref{b2}.
For each $\d\geq0$ and $0<\b\leq\oo$ one may define the
critical value $\l_\crit=\l_\crit(\d,\b)$ by
\[
\begin{split}
\l_\crit(\d,\b)&=
\sup\big\{\l>0:\limsup_{N\rightarrow\oo}\chi_\L(\l,\d,\b)<\oo\big\},
\quad \mbox{if }\b<\oo,\\
\l_\crit(\d,\oo)&=\sup\big\{\l>0:
\limsup_{N,\b\rightarrow\oo}\chi_\L(\l,\d,\b)<\oo\big\}.
\end{split}
\]
This is the the same critical value as referred to in
Section~\ref{intro_sec}, see~\cite[Theorem~1.1]{bjogr}.
In all simultaneous limits $N,\b\rightarrow\oo$ we assume
that the limit is taken so that $\b$ and $N$ are of the 
same order (this is convenient when using the graphical representation
and is related to `van Hove convergence').
It is well-known that $0<\l_\crit<\oo$
provided either $\b<\oo$ and
$d\geq 2$, or $\b=\oo$ and $d\geq1$.

For each $0<\b\leq\oo$ and $\l<\l_\crit$ there
is a unique infinite-volume limit measure of $\mu^\b_\L$;
for $\b=\oo$ it is obtained in the simultaneous
limit $N,\b\rightarrow\oo$.  For simplicity we will
denote this limit measure by $\mu$.
From the dominated convergence theorem it follows that,
as $N\rightarrow\oo$ or $N,\b\rightarrow\oo$,
the limits $\chi$ and $B$ of $\chi_\L$ and $B_\L$, respectively,
exist, and that
\[
\chi=\left\{\begin{array}{ll}
\sum_{x\in\ZZ^d}\int_0^\b
\mu\big(\s(0,0)\s(x,t)\big)\,dt, &
\mbox{if } \b<\oo,\\
\sum_{x\in\ZZ^d}\int_{-\oo}^\oo
\mu\big(\s(0,0)\s(x,t)\big)\,dt, &
\mbox{if } \b=\oo,
\end{array}\right.
\]
and
\[
B=\left\{\begin{array}{ll}
\sum_{x\in\ZZ^d}\int_0^\b
\mu\big(\s(0,0)\s(x,t)\big)^2\,dt, &
\mbox{if } \b<\oo,\\
\sum_{x\in\ZZ^d}\int_{-\oo}^\oo
\mu\big(\s(0,0)\s(x,t)\big)^2\,dt, &
\mbox{if } \b=\oo.
\end{array}\right.
\]
By~\cite[Theorem~6.3]{bjogr}, $\chi\uparrow\oo$
as $\l\uparrow\l_\crit$.  Since 
$\mu\big(\s(0,0)\s(x,t)\big)\leq1$ we have  $B\leq\chi$
and hence $B$ may or may not diverge at $\l_\crit$.
In the following statement, $\log_+ \chi$ is shorthand
for $(\log \chi)\vee 0$.

\begin{lemma}\label{bub_lem2}
Let $0<\b\leq\oo$
and assume that $\d\geq\eps_0$ and $\eps_0\leq\l<\l_\crit$
for some fixed $\eps_0>0$.
\begin{enumerate}
\item Suppose either (a) $\b<\oo$ and $d>4$, or (b)
$\b=\oo$ and $d>3$.  Then
there is a finite constant $C_1=C_1(\eps_0,\b)$ such that
$B(\l,\d,\b)\leq C_1$.
\item Suppose either (a) $\b<\oo$ and $d=4$, or (b)
$\b=\oo$ and $d=3$.  Then
there is a finite constant $C_2=C_2(\eps_0,\b)$ such that
$B(\l,\d,\b)\leq C_2(1+\log_+\chi)$.
\end{enumerate}
\end{lemma}
\begin{proof}
We start by proving the first statement.
Write 
$\hat K=\big(\tfrac{2\pi}{2N}\L\big)\times\big(\tfrac{2\pi}{\b}\ZZ\big)$
and $\hat K^\times=\hat K\setminus{(0,0)}$.
By Plancherel's formula
and the infrared bound,
\[
\begin{split}
B_\L&=\frac{1}{\b|\L|}
\sum_{(k,l)\in\hat K}\hat c_\L(k,l)^2
=\frac{1}{\b|\L|}\Big[\chi_\L^2+
\sum_{(k,l)\in\hat K^\times}\hat c_\L(k,l)^2\Big]\\
&\leq\frac{1}{\b|\L|}\Big[\chi_\L^2+
\sum_{(k,l)\in\hat K^\times}
\Big(\frac{48}{2\l\hat L(k)+l^2/2\d}\Big)^2\Big].
\end{split}
\]
We conclude that for all $\l<\l_\crit$, the 
bubble diagram $B$ is at most
\begin{equation}\label{bub1}
\frac{1}{(2\pi)^d\b}\sum_{l\in\frac{2\pi}{\b}\ZZ}
\int_{(-\pi,\pi]^d}\Big(\frac{48}{2\l\hat L(k)+l^2/2\d}\Big)^2dk,
\quad\mbox{if }\b<\oo,
\end{equation}
and at most
\begin{equation}\label{bub2}
\frac{1}{(2\pi)^{d+1}}\int_{-\oo}^\oo\int_{(-\pi,\pi]^d}
\Big(\frac{48}{2\l\hat L(k)+l^2/2\d}\Big)^2dl\,dk,
\quad\mbox{if }\b=\oo.
\end{equation}
For any $a>0$,
\begin{equation}\label{int1}
\sum_{l\in\ZZ}\frac{1}{(a+l^2)^2}\leq
\frac{1}{a^2}+\int_{-\oo}^\oo\frac{1}{(a+l^2)^2}\,dl=
\frac{1}{a^2}+\frac{\pi/2}{a^{3/2}}.
\end{equation}
Applying this with $a(k)=4\l\d\hat L(k)$ we deduce that
the quantity in~\eqref{bub1} is at most an absolute constant times
\begin{equation}\label{bub3}
\frac{\d^2}{\b}\int_{(-\pi,\pi]^d}\Big(\frac{1}{a(k)^2}+
\frac{1}{a(k)^{3/2}}\Big)dk
\end{equation}
and that the quantity in~\eqref{bub2} is at most an absolute constant 
times
\[
\d^2\int_{(-\pi,\pi]^d}\frac{1}{a(k)^{3/2}}dk.
\]
Recall that
$\hat L(k)=\sum_{j=1}^d(1-\cos(k_j))$.
In particular, $\hat L(k)\rightarrow0$ as $k\rightarrow0$
but $\hat L(k)$ is positive for all nonzero
$k\in(-\pi,\pi]^d$.  The dominant term
in~\eqref{bub3} as $k\rightarrow0$ is $1/a(k)^2$.
Hence there is a constant $c=c(\eps_0,\b)$ such that 
\[
B\leq c\int_{(-\pi,\pi]^d}\frac{dk}{\hat L(k)^\a},
\]
where $\a=2$ if $\b<\oo$ and $\a=3/2$ if $\b=\oo$.
There is also a constant $c'=c'(d)$ such that
$\hat L(k)\geq c'\|k\|_2^2$ for all $k\in(-\pi,\pi]^d$.
By using polar coordinates we see that
\[
\int_{(-\pi,\pi]^d}\frac{dk}{\|k\|_2^{2\a}}\leq
\int_0^{2\pi} \frac{r^{d-1}}{r^{2\a}}dr,
\]
which is finite if $d>2\a$.  Thus for $d>2\a$ and 
for all $\l<\l_\crit$, we have that
$B\leq C_1(\eps_0,\b)=c(\eps_0,\b)/c'(d)\int_0^{2\pi}r^{d-1-2\a}dr<\oo$.
This proves the first statement.

We now prove the second statement.  
By the triangle inequality and the nonnegativity
of $\hat c(k,l)$ we have that $|\hat c_\L(k,l)|\leq\chi_\L$
for all $(k,l)\in\hat K$.  Thus
\begin{equation*}
|\hat c_\L(k,l)|\leq
\min\Big(\chi_\L,\frac{48}{2\l\hat L(k)+l^2/2\d}\Big)
\leq \frac{2}{\chi_\L^{-1}+(2\l\hat L(k)+l^2/2\d)/48}.
\end{equation*}
Hence there is a constant $c=c(\eps_0,\b)$ such that
\begin{equation}\label{bub4}
B\leq c\sum_{l\in\frac{2\pi}{\b}\ZZ}
\int_{(-\pi,\pi]^d}\Big(\frac{1}{\chi^{-1}+\hat L(k)+l^2}\Big)^2dk,
\quad\mbox{if }d=4\mbox{ and }\b<\oo,
\end{equation}
and 
\begin{equation}\label{bub5}
B\leq c\int_{-\oo}^\oo\int_{(-\pi,\pi]^d}
\Big(\frac{1}{\chi^{-1}+\hat L(k)+l^2}\Big)^2dl\,dk,
\quad\mbox{if }d=3\mbox{ and }\b=\oo.
\end{equation}
As in the first part it follows that there is a constant
$c'(\eps_0,\b)$ such that
\begin{equation*}
B\leq c' \int_0^{2\pi} \frac{r^3}{(\chi^{-1}+r^2)^2}dr,
\quad\mbox{if }d=4\mbox{ and }\b<\oo,
\end{equation*}
and
\begin{equation*}
B\leq c' \int_0^{2\pi} \frac{r^2}{(\chi^{-1}+r^2)^{3/2}}dr,
\quad\mbox{if }d=3\mbox{ and }\b=\oo.
\end{equation*}
By explicit computation of these integrals (or otherwise)
the result follows.
\end{proof}

\subsection{Critical exponents}
We now turn to the proof of Theorem~\ref{mf_thm}.
We split the result into two propositions, one for
the upper bound and one for the lower bound, with some
additional details added.
Note that $\chi$ is (weakly) increasing in $\l$ and 
decreasing in $\d$ (this can be seen, for example,
in~\eqref{l1},~\eqref{rp2} and~\eqref{rp3}).
In particular, $\l_\crit(\d)$ is weakly increasing in $\d$.
Recall also that $\chi(\l,\d)\uparrow\oo$ as $\l\uparrow\l_\crit(\d)$.

The lower bound does not depend on Lemma~\ref{bub_lem2}
and is valid whenever  $0<\l_\crit<\oo$, which we recall is the
case whenever either $d\geq2$, or $d=1$ and $\b=\oo$.
Here and in what follows $\|(a,b)-(c,d)\|$ denotes the
Euclidean distance between points $(a,b),(c,d)\in\RR^2$.
\begin{proposition}\label{lb_prop}
Suppose that $0<\l_\crit<\oo$.  Then
\begin{enumerate}
\item $\chi$ is continuous in $(\l,\d)$ whenever
$\l<\l_\crit(\d)$;
\item For all $\d_0\geq0$ and all $\d\geq\d_0$ and
$0\leq\l\leq\l_\crit(\d_0)$ we have that
\begin{equation*}
\chi(\d,\l)\geq \frac{1/(\sqrt{3}(4d+2))}{\|(\d,\l)-(\d_0,\l_\crit(\d_0))\|}.
\end{equation*}
\end{enumerate}
\end{proposition}
\begin{proof}
We use the upper bounds in Lemma~\ref{diff_lem}.
These may be rewritten as
\begin{equation*}
-\frac{\partial\chi_\L^{-1}}{\partial\l}\leq 4d
\qquad\mbox{and}\qquad 
\frac{\partial\chi_\L^{-1}}{\partial\d}\leq 2
\end{equation*}
Let $0\leq\d_1<\d_2$ and $0\leq\l_1<\l_2$.  Then
\begin{equation*}
\begin{split}
\chi_\L^{-1}(\d_2,\l_1)-\chi_\L^{-1}(\d_1,\l_2)&=
\int_{\d_1}^{\d_2}\frac{\partial\chi_\L^{-1}}{\partial\d}d\d-
\int_{\l_1}^{\l_2}\frac{\partial\chi_\L^{-1}}{\partial\l}d\l\\
&\leq (4d+2)(\d_2-\d_1+\l_2-\l_1)\\
&\leq\sqrt{3}(4d+2)\|(\d_2,\l_1)-(\d_1,\l_2)\|.
\end{split}
\end{equation*}
Now let $N\rightarrow\oo$ or $N,\b\rightarrow\oo$
as approproate.  The continuity statement follows immediately,
and the second statement follows on letting
$\d_1=\d_0$ and $\l_2\uparrow\l_\crit(\d_0)$.
\end{proof}

We now turn to the upper bound in Theorem~\ref{mf_thm}.
We will in what follows
assume that $\d,\l>\eps_0$
for some arbitrary but fixed $\eps_0>0$.
By continuity we have the following.
\begin{lemma}\label{nhd_lem}
For each $\d_0>\eps_0$ and each $C>0$
there is an open, bounded neighbourhood $U$ 
of $(\d_0,\l_\crit(\d_0))$ such
that $\chi(\d,\l)\geq 2 C$ for all $(\d,\l)\in U$.
\end{lemma}

\begin{proposition}\label{ub_prop}
Fix $\theta\in(0,\oo)$ and $\d_0>0$.  
\begin{enumerate}
\item Suppose that either $\b<\oo$ and $d>4$,
or $\b=\oo$ and $d>3$.
There is a neighbourhood $U$ of $(\d_0,\l_\crit(\d_0))$
and a constant $c(\d_0,\theta)$
such that for all
$(\d,\l)\in U$, of the form $\d=\d'-t$, $\l=\l_\crit(\d')+\theta t$
with $t<0$, we have that $\chi(\l,\d)\leq c(\d_0,\theta)/|t|$.
\item Suppose that either $\b<\oo$ and $d=4$,
or $\b=\oo$ and $d=3$.  
There is a neighbourhood $U$ of $(\d_0,\l_\crit(\d_0))$
and a constant $c(\d_0,\theta)$
such that for all
$(\d,\l)\in U$, of the form $\d=\d'-t$, $\l=\l_\crit(\d')+\theta t$
with $t<0$, we have that 
$\chi(\l,\d)\leq c(\d_0,\theta)|\log t|/|t|$.
\end{enumerate}
\end{proposition}
\begin{proof}
Throughout
the proof we will let $(\d,\l)$ be of the form $\d=\d'-t$, 
$\l=\l_\crit(\d')+\theta t$ with $t<0$, and will write 
$\chi_\L(t)$ and $B_\L(t)$ for $\chi_\L(\d,\l)$
and $B_\L(\d,\l)$, respectively.  We have that the derivative in $t$
\[
\chi_\L'(t)=\theta\frac{\partial\chi_\L}{\partial\l}+
\Big(-\frac{\partial\chi_\L}{\partial\d}\Big).
\]
From~\eqref{di1} and~\eqref{di2} we see that
\begin{equation}\label{f2}
\chi_\L'(t)\geq(4d\theta+2)\chi_\L(t)^2\frac{1-B_\L(t)/\chi_\L(t)}
{1+(2d\l +4\d +\l/\theta+8d\d\theta) B_\L(t)}.
\end{equation}
Restricting $(\d,\l)$ to an arbitrary bounded open set $U'$ 
containing $(\d_0,\l_\crit(\d_0))$ we may
replace $2d\l +4\d +\l/\theta+8d\d\theta$ by a uniform
upper bound $c_1(\d_0,\theta)$.
Since 
\[
\chi_\L'(t)/\chi_\L(t)^2=-\frac{d}{dt}\frac{1}{\chi_\L(t)}
\]
it follows on integrating that for all $t_1<t_2<0$ (such that 
the corresponding points $(\d,\l)$ lie in $U'$) we have
\[
\frac{1}{\chi_\L(t_1)}-\frac{1}{\chi_\L(t_2)}\geq (4d\theta+2)
\int_{t_1}^{t_2}\frac{1-B_\L(t)/\chi_\L(t)}{1+c_1(\d_0,\theta) B_\L(t)}\,dt.
\] 
Letting $\L\uparrow\ZZ^d$ and applying Fatou's Lemma, we see that
\begin{equation}\label{fat}
\frac{1}{\chi(t_1)}-\frac{1}{\chi(t_2)}\geq (4d\theta+2)
\int_{t_1}^{t_2}\frac{1-B(t)/\chi(t)}{1+c_1(\d_0,\theta) B(t)}\,dt.
\end{equation}

For the first part of the statement, let $U$ be as in
Lemma~\ref{nhd_lem} with $C=C_1(\eps_0,\b)$ of Lemma~\ref{bub_lem2}.
From~\eqref{fat} we see that
\begin{equation*}
\frac{1}{\chi(t_1)}-\frac{1}{\chi(t_2)}\geq 
(t_2-t_1)\frac{2d\theta+1}{1+c_1(\d_0,\theta)C_1(\eps_0,\b)}.
\end{equation*}
Letting $t_2\uparrow0$ gives the result.

For the second part, we deduce from~\eqref{fat}
and the second part of Lemma~\ref{bub_lem2} that
\begin{equation*}
\frac{1}{\chi(t_1)}-\frac{1}{\chi(t_2)}\geq (4d\theta+2)
\int_{t_1}^{t_2}\frac{1-C_2(\eps_0,\b)(1+\log\chi(t))/\chi(t)}
{1+c_2(\d_0,\theta) C_2(\eps_0,\b)(1+\log\chi(t))}\,dt.
\end{equation*}
It follows that there is a constant $c_2(\eps_0,\d_0,\theta,\b)$
such that
in a small enough neighbourhood $U$ of $(\d_0,\l_\crit(\d_0))$ 
we have
\begin{equation*}
\frac{1}{\chi(t_1)}-\frac{1}{\chi(t_2)}\geq c_2
\int_{t_1}^{t_2}\frac{1}{1+\log\chi(t)}\,dt.
\end{equation*}
The second part now follows from~\cite[Lemma~4.1]{ag}.
\end{proof}

\begin{proof}[Proof of Theorem~\ref{mf_thm}]
The lower bounds are immediate from Proposition~\ref{lb_prop}.
The upper bounds follow from Proposition~\ref{ub_prop}
on noting that $\|(\d,\l)-(\d_0,\l_\crit(\d_0))\|=\sqrt{1+\theta^2}\cdot |t|$.
\end{proof}

\begin{remark}\label{gamma_rk}
A similar argument as in
Proposition~\ref{ub_prop} would give the critical exponent value
$\g=1$ also for `vertical' ($\d$ constant) and `horizontal'
($\l$ constant) approach to the critical curve,
\emph{subject to} first proving differential inequalities of the form
\[
\frac{\partial\chi_\L}{\partial \l}\leq c_1
\Big(-\frac{\partial\chi_\L}{\partial \d}\Big)\qquad\mbox{and}
\qquad -\frac{\partial\chi_\L}{\partial \d}\leq c_2
\frac{\partial\chi_\L}{\partial \l},
\]
with $c_1,c_2$ uniform in $\L$.  We do not pursue this here,
only noting that
similar inequalities hold for classical Ising and Potts 
models~\cite{aiz_grimm,bez_gr_kes}.
In the absence of such inequalities it is a-priori possible
that the behaviour for vertical or horizontal approach
differs from the case in Theorem~\ref{mf_thm}.
For example, if it were the case 
that $\l_\crit(\d)\sim(\d-\d_0)^2$ as $\d$
decreases to some $\d_0>0$, then
Proposition~\ref{ub_prop} would imply that 
$\chi(\d,\l_\crit(\d_0))\sim(\d-\d_0)^{-2}$.
\end{remark}

\bibliography{irb}
\bibliographystyle{plain}

\end{document}